\documentclass[a4paper,twoside,12pt]{article}      
\usepackage{amsmath,amssymb,amsfonts} 
\usepackage{amsthm}
\usepackage{graphics}                 
\usepackage{color}                    
\usepackage{hyperref}                 
\usepackage{graphicx,eucal}
\usepackage{latexsym,epsfig,epic,eepic}

\newcommand{\bbD}{\mathbb{D}}

\newcommand{\bbN}{\mathbb{N}}
\newcommand{\bbZ}{\mathbb{Z}}
\newcommand{\bbQ}{\mathbb{Q}}
\newcommand{\bbR}{\mathbb{R}}
\newcommand{\bbC}{\mathbb{C}}

\newcommand{\bbP}{\mathbb{P}}

\newcommand{\mcF}{\mathcal{F}}

\newcommand{\eps}{\varepsilon}

\newcommand{\Var}{\mathop{\mathrm{Var}}}
\newcommand{\dist}{\mathop{\mathrm{dist}}\nolimits}

\newcommand{\NE}{\mathrm{NE}}
\newcommand{\Diff}{\mathrm{Diff}}
\newcommand{\id}{\mathrm{id}}
\newcommand{\Dirac}{\mathrm{Dirac}}

\newtheorem{theorem}{Theorem}[section]
\newtheorem{thmain}{Theorem}
\newtheorem{proposition}[theorem]{Proposition}
\newtheorem{corollary}[theorem]{Corollary}
\newtheorem{lemma}[theorem]{Lemma}
\newtheorem{question}[theorem]{Question}
\newtheorem{conjecture}[theorem]{Conjecture}
\newtheorem{example}[theorem]{Example}

\theoremstyle{definition}
\newtheorem{definition}[theorem]{Definition}
\newtheorem{remark}[theorem]{Remark}

\newcommand{\ctg}{\mathop{\mathrm{ctg}}}
\newcommand{\arcctg}{\mathop{\mathrm{arcctg}}}
\newcommand{\supp}{\mathop{\mathrm{supp}}}

\newcommand{\wtp}{\widetilde{\varphi}}

\newcommand{\Ls}{\mbox{\rm($\Lambda\star$)}}
\newcommand{\smin}{\mbox{\rm($\star$)}}
\newcommand{\Leb}{\mathrm{Leb}}
\newcommand{\PSL}{\mathrm{PSL}}

\date{\today}
\title{On the question of ergodicity for minimal group actions on the circle}

\author{Bertrand Deroin, Victor Kleptsyn, Andr\'es Navas}

\begin{document}
\maketitle
\begin{abstract}
This work is devoted to the study of minimal, smooth actions of finitely generated groups 
on the circle. We provide a sufficient condition for such an action to be ergodic (with
respect to the Lebesgue measure), and we illustrate this condition by studying two
relevant examples. Under an analogous hypothesis, we also deal with the problem of the
zero Lebesgue measure for exceptional minimal sets. This hypothesis leads to many other interesting conclusions, mainly concerning the stationary and conformal measures. Moreover, several questions are left open. The methods work as well for codimension-one foliations, though the results for this case are not explicitly stated.
\end{abstract}

\tableofcontents

\section{Introduction}

\subsection{Minimality and ergodicity}

An invariant probability measure $\mu$ (for a map, or for a group action) is said to be \emph{ergodic}
if every invariant measurable set is either of zero or full $\mu$-measure. This is equivalent to the
fact that every invariant measurable function is constant $\mu$-a.e., and also to the fact that
$\mu$ is an extremal point of the compact, convex set formed by the invariant probability measures.

The definition of ergodicity can be naturally extended to non necessarily invariant
measures $\mu$ which are at least \emph{quasi-invariant}, that is, such that $g_*\mu$
is absolutely continuous with respect to $\mu$ for every element $g$ in the acting group. (To simplify the exposition, all the measures in this article are supposed to be probability measures.)

\begin{definition}
Let $\mu$ be a measure on a measurable space $X$ which is quasi-invariant by 
the action of a group $G$. We say that $\mu$ is \emph{ergodic} if for every measurable 
$G$-invariant subset $A\subset X$ either $\mu(A)=0  \text{ or }  \mu(X\setminus A)=0.$
\end{definition}

Notice that the definition of ergodicity concerns both the action and the measure. 
However, for several actions an invariant (or quasi-invariant) measure is naturally 
defined. For instance, symplectic maps or Hamiltonian flows and their restrictions 
to fixed energy levels have natural invariant measures, and for 
any $C^1$ diffeomorphism the Lebesgue measure is quasi-invariant. In these 
situations, one focuses on the action itself, and the ergodicity is always meant 
with respect to this natural measure.

Ergodicity can be thought of as a property involving some complexity for the orbits of the
action coming from the fact that this action is irreducible from a measurable point of
view. The topological counterpart to this notion corresponds to \emph{minimality}:

\begin{definition}
A continuous action of a group $G$ on a topological space $X$ is said to be
\emph{minimal} if for every $G$-invariant closed subset $A\subset X$ either
$A=\emptyset \, \text{ or } \, A=X$. Equivalently, an action is minimal
if all of its orbits are dense.
\end{definition}

It is natural to ask until what extend the properties of ergodicity and minimality are related.
In one direction, it is easy to see that, in general, the former does not imply the latter.
Indeed, ergodicity concerns the behavior of almost every point, and not of \emph{all} the points.
Actually, one can easily construct examples of ergodic group actions having global
fixed points. The question in the opposite direction is more interesting:

\begin{question}
Under what conditions a smooth action of a group on
a compact manifold is necessarily Lebesgue-ergodic?
\end{question}

The following widely known conjecture concerns the one-dimensional case of this question.
The main result of this work, namely Theorem~\ref{description} later on, allows to solve it in the 
affirmative under some additional assumptions that seem to us interesting and sufficiently mild.

\begin{conjecture}
Every minimal smooth action of a finitely generated group on the circle is ergodic with respect
to the Lebesgue measure.\footnote{One may also ask about ergodicity for smooth actions of
finitely generated groups on the circle or the interval having a dense orbit. However,
we will not deal with this more general question in this work.}
\label{conj}
\end{conjecture}

Conjecture~\ref{conj} has been answered by the affirmative in many cases by the exponential expansion strategy (largely going back to D. Sullivan). We will recall this strategy in Section~\ref{s:background}. Here we content ourselves in recalling the definition of Lyapunov expansion exponent in order to state the main result which is known in this direction.  For simplicity of the exposition, from 
now on we assume that the diffeomorphisms in our group~$G$ 
preserve the orientation. This assumption is non-restrictive, as otherwise one can pass (without 
loosing the minimality) to the index-two subgroup formed by the orientation preserving elements.

\begin{definition}
Let~$G$ be a finitely generated group of circle diffeomorphisms. Let 
$\mcF=\{g_1,\dots,g_k\}$ be a finite set of elements generating $G$ 
as a semigroup, and let $\left\|\cdot\right\|_{\mcF}$ be the corresponding 
word-length norm. The \emph{Lyapunov expansion exponent} of $G$ at a point $x\in S^1$ is
$$\lambda_{exp}(x;\mcF):=\limsup_{n\to\infty} \max_{\|g\|_{\mcF}\le n} \frac{1}{n} \, \log \big( g'(x) \big).$$
\end{definition}

Notice that the value of the Lyapunov expansion exponent~$\lambda_{\exp}(x;\mcF)$ 
depends on the choice of the finite system of generators
$\mathcal{F}$. However, the fact that this number is equal
to zero or is positive does not depend on this choice. Thus, relations of the form $\lambda_{exp}(x)=0$
or $\lambda_{exp}(x)>0$ make sense, although the number $\lambda_{exp}(x)$ is not well-defined without referring to~$\mcF$.

\begin{theorem}[S.~Hurder~\cite{Hurder}]\label{thm:expansion}
If $G$ is a finitely generated subgroup of $\Diff_+^{1+\alpha}(S^1)$ acting minimally, then the
Lyapunov expansion exponent is constant Lebesgue-almost everywhere. If this constant is positive,
then the action is ergodic.
\end{theorem}

This constant is called the Lyapunov expansion exponent $\lambda_{\exp}(G;\mcF)$ of the action. As before, it depends on the particular choice of the system of generators~$\mcF$, though the relations of the form $\lambda_{\exp}(G)>0$ or $\lambda_{\exp}(G)=0$ make sense without referring to~$\mcF$.

A very simple compactness type argument shows that $\lambda_{\exp}(G)>0$ in the case
where for all $x\in S^1$ there exists $g\in G$ such that $g'(x)>1$. Actually, the 
ergodicity of minimal $C^{1 + \alpha}$ actions satisfying the latter condition was 
proved earlier in~\cite{Navas-ens}. 

All of this serves as a good motivation for the following

\begin{definition}
A point $x \!\in\! S^1$ is said to be \emph{non-expandable} if for all $g\in G$
one has $g'(x)\le 1$. 
\end{definition}

One should immediately point out that the presence of non-expandable points 
does not contradict the minimality. For instance, the canonical action 
of the modular group $\PSL_2(\bbZ)$ on $S^1=\bbP(\bbR^2)$ is minimal, 
but the points $(0:1)$ and $(1:0)$ are non-expandable (see Section \ref{ss:PSL-2-Z}). 
Another example is provided by the smooth actions of Thompson group~$T$ 
constructed by \'E.~Ghys and V.~Sergiescu (see Section~\ref{ss:Thompson}). 
However, the non-expandable points represent a (potential) obstacle for performing 
the exponential expansion strategy.

Denote the set of non-expandable points by $\NE = \NE(G)$. Notice that this 
set depends on the coordinates chosen on the circle. Thus, we suppose a metric 
on the circle to be fixed, and we will discuss the dependence of the $\NE$-set on the 
metric later (see Corollary~\ref{cor:independence}). The assumption of 
our main result is given by the next


\begin{definition}
The group $G$ satisfies \emph{property~\smin} if it acts minimally and for 
every $x\in \NE(G)$ there exist $g_+,g_-$
in $G$ such that $g_+(x)=g_-(x)=x$ and $g_+$ (respectively,~$g_-$) has no other
fixed point in some interval $(x,x+\eps)$ (respectively, $(x-\eps,x)$).
\end{definition}

\begin{remark} 
Notice that property~\smin\ holds if for every $x\in \NE$ there exists an element $g\in G$
such that~$x$ is an isolated fixed point of~$g$. In particular, if $G$ is a group of
real-analytic diffeomorphisms, then property~\smin\ is equivalent to:
$$\mbox{for all } x \in \NE(G)  \mbox{ there exists }
g \in G \mbox{ such that } g \neq id \mbox{ and } g(x)=x.$$
Indeed, every fixed point of a nontrivial analytic diffeomorphism is isolated.
\end{remark}

We are now ready to state the main result of this paper.
In order to simplify our discussion\footnote{The reader will easily check that, as it is usual in the subject,
the methods and results in this work also apply to groups of diffeomorphisms of class $C^1$ having
Lipschitz derivative.}, we will only deal with finitely generated groups of circle diffeomorphisms that are of class~$C^2$.

\begin{thmain}\label{description} 
If $G$ is a finitely generated subgroup of
$\Diff^2_+(S^1)$ satisfying property~\smin, then the following hold:
\begin{enumerate}
\item $\NE(G)$ is finite.
\item For every point $x\in S^1$ either the set of derivatives $\{g'(x)\mid g\in G\}$
is unbounded, or $x$ belongs to the orbit of some non-expandable point.
\item $G$ is ergodic with respect to the Lebesgue measure.
\end{enumerate}
\end{thmain}
The second conclusion of Theorem~\ref{description} allows deducing the following 
\begin{corollary}\label{cor:independence}
For finitely generated groups of $C^2$ circle diffeomorphisms,
property~\smin\ does not depend on the choice of the Riemannian metric on the circle.
\end{corollary}

The assumption in Theorem A is well illustrated by two fundamental examples. The first one 
appears in~\cite{Ghys-Sergiescu}, where \'E.~Ghys and V.~Sergiescu showed that the
canonical (and actually unique up to semiconjugacy: see \cite{Liousse}) action of
Thompson group~$T$ on the circle is topologically conjugate to an action by $C^{\infty}$
diffeomorphisms.\footnote{They also constructed a $C^{\infty}$-action of~$T$ with a minimal invariant
Cantor set which is semi-conjugate to the standard one, obtaining as a corollary the rationality of
the rotation number for each element of~$T$. As a consequence, the ergodicity for the smooth
minimal actions of~$T$ on the circle cannot be deduced from Katok-Herman's result discussed
in Section~\ref{ss:minimality}.\label{f:G-S-rational}} For this (minimal) action (that we recall in Section~\ref{ss:Thompson}) 
we prove the following 

\begin{thmain}\label{Thompson}
The $\NE$-set for the minimal Ghys-Sergiescu's action of Thompson group 
T consists of a single point, which is an isolated fixed point of an element. (Therefore, this action
satisfies property~\smin.) On the other hand, the Lyapunov expansion exponent of the action is zero.
\end{thmain}

It was pointed out to us by  \'E.~Ghys that there exist smooth actions of the group~$T$ (still satisfying  property~\smin\ and with zero Lyapunov expansion exponent) having more than one non-expandable point: see Remark~\ref{etienne}.

The second example of a minimal group action with a non-empty set of non-expandable 
points is the already mentioned action of $\PSL_2(\bbZ)$. 
(Notice that, since this group is discrete inside $\PSL_2(\bbR)$, the rotation number of each of its elements is rational: compare footnote~\ref{f:G-S-rational}.) For this case we have the following

\begin{thmain}\label{PSL2Z}
The only non-expandable points of the canonical action of $\PSL_2(\bbZ)$ (in the 
standard affine chart) are $0$ and $\infty$. Both of them are isolated fixed points of certain elements 
(namely, $x\mapsto x/(x+1)$ and $x\mapsto x+1$, respectively), and therefore the action satisfies
property~\smin. On the other hand, its Lyapunov expansion exponent equals zero.
\end{thmain}

The fact that the action of $\PSL_2(\bbZ)$ is ergodic is well-known. Indeed, one
way to  show this consists in extending this action inside the Poincar\'e disc~$\bbD$.
If a measurable invariant set~$A$ existed, then the solution of the Dirichlet problem
with $\mathbf{1}_A$ as boundary value would be an invariant harmonic function. This function 
would then descend to the quotient $\mathbb{D}/\PSL_2(\bbZ)$, which is the modular 
surface. However, there exists no non-constant, bounded, harmonic function on the 
modular surface, which gives a contradiction.
Despite this very simple argument, it is
interesting to notice that the ergodicity 
can be also deduced from the property~\smin.

\subsection{Exceptional minimal sets}

Recall that for every group of circle homeomorphisms without finite orbits and whose action is not 
minimal, there exists a minimal closed invariant set which is homeomorphic to the Cantor set (and
which is commonly called an {\em exceptional minimal set}): see for instance \cite{Ghys-actions}. 
The following conjecture, stated by G.~Hector (as far as we know, in 1977-78), appears to be 
fundamental in this context:

\begin{conjecture}[G.~Hector] 
If a finitely generated group of $C^2$ circle
diffeomorphisms admits an exceptional minimal set $\Lambda$, 
then the Lebesgue measure of $\Lambda$ is zero. \label{conj2}
\end{conjecture}

In this work, we deal with this conjecture under a condition which is analogous to property~\smin:

\begin{definition}
Let $G$ be a finitely generated group of circle diffeomorphisms admitting an exceptional 
minimal set~$\Lambda$. We say that~$G$ satisfies 
\emph{property~\Ls}, if for every $x\in \Lambda\bigcap NE$ there exist $g_{-}, 
g_{+}$ in $G$ such that $g_+(x)=g_-(x)=x$ and $g_+$ (respectively, $g_-$)
has no other fixed point in some interval $(x,x+\eps)$ (respectively, $(x-\eps,x)$).\footnote{Actually, for the case where $x$ is isolated in $\Lambda$ from one side this condition may be weakened, only asking for an element having~$x$ as an isolated fixed point from the side where it is an accumulation point of $\Lambda$.}
\end{definition}

The theorem below is a natural analogue of our Theorem~\ref{description} for exceptional 
minimal sets. As their proofs are also completely analogous, we leave to the reader the task 
of adapting the arguments of the proof of Theorem~\ref{description} to this case. 

\begin{thmain}\label{Cantor}
Let $G$ be a finitely generated group of $C^2$ circle diffeomorphisms
having an exceptional minimal set $\Lambda$. If the action satisfies the property~\Ls, then:
\begin{enumerate} 
\item The set $\Lambda\cap \NE(G)$ is finite.
\item For each $x \in \Lambda$ not contained in the orbit of $\NE(G)$, 
the set of derivatives $\{g'(x) \!: x \in G\}$ is unbounded.
\item The Lebesgue measure of $\Lambda$ is zero.
\end{enumerate}
\end{thmain}

As in the case of minimal actions, the second conclusion of the theorem above implies the following

\begin{corollary}
For finitely generated groups of $C^2$ circle diffeomorphisms having an exceptional 
minimal set~$\Lambda$, property~\Ls{} does not depend on the choice of the Riemannian 
metric on the circle. 
\end{corollary}


\subsection{Conformal measures}

For conformal (in particular, for one-dimensional, smooth) maps, a fundamental property of the Lebesgue measure is that its infinitesimal change at a point is given by the derivative to the 
power of the dimension of the underlying space. This property was 
generalized by D. Sullivan (see~\cite{Sullivan}), who introduced the 
concept of \emph{conformal measure} as a powerful tool for studying the 
dynamics on exceptional minimal sets.

\begin{definition}
Let $G$ be a group of conformal transformations. A measure $\mu$ on the underlying 
space is said to be \emph{conformal with exponent~$\delta$} (or simply \emph{$\delta$-conformal}), if for every Borel set~$B$ and for every map $g\in G$ one has
$$
\mu (g(B)) = \int _B |g'(x)|^{\delta} d\mu(x),
$$ 
where $|g'|$ stands for the scalar part of the (conformal) derivative of~$g$.
Equivalently, for $\mu$-almost every point $x$, the Radon-Nikodym derivative of $g_*\mu$ w.r.t.~$\mu$ equals
$$
\frac{dg_* \mu}{d\mu}(x)= \frac{1}{|g'(g^{-1}(x))|^{\delta}}.
$$
\end{definition}

For the case of conformal maps, the conformal exponent of the Lebesgue measure equals
the dimension. Nevertheless, in presence of a proper closed invariant set, one can ask for 
the existence of a conformal measure (perhaps with a different exponent)  supported on it. 
For the case of a subgroup of~$\PSL_2(\bbC)$ acting conformally on the Riemann sphere 
with a proper minimal closed invariant set different from a finite orbit, such a measure 
was constructed by D.~Sullivan in~\cite{Sullivan1}. 

It is unclear whether for every finitely generated group of circle diffeomorphisms having 
an exceptional minimal set $\Lambda$, there exists a conformal measure supported 
on~$\Lambda$. This is the case for groups of real-analytic diffeomorphisms. Indeed, 
in this case the group acts discretely on the complement of $\Lambda$  (see~\cite{Hector}). 
Then, the arguments used by D.~Sullivan in~\cite{Sullivan1} apply to prove the existence of an 
exponent $0 < \delta \leq 1$ for which a $\delta$-conformal measure exists.

Conformal measures are expected to be ergodic as the Lebesgue measure 
is supposed to be in the minimal case. Although we are not able to settle this problem 
in its full generality here, we are able to deal with a weaker property, namely \emph{conservativity}:
\begin{definition}
An action of a group~$G$ on a measurable space $X$ that quasi-preserves a 
measure $\mu$ is said to be \emph{conservative}, if for every set~$A$ with $\mu(A)>0$, there 
exists an element $g\in G\setminus \{e\}$ such that $\mu(A\cap g(A))>0$.
\end{definition}
A non-conservative action is indeed automatically non-ergodic, since for each measurable subset~$B\subset A$ of intermediate measure $0<\mu(B)<\mu(A)$ the set $G(B)=\bigcup_{g\in G} g(B)$ is invariant and has intermediate measure $0<\mu(G(B))<1$.

The following result, due to D.~Sullivan, can be viewed as a partial evidence supporting 
Conjecture~\ref{conj}:

\begin{theorem} [D.~Sullivan~\cite{Sullivan}]\label{thm:conservativity} 
Let $G$ be a finitely generated group of $C^2$ circle diffeomorphisms. If
the action of $G$ is minimal, then it is conservative (with respect to the 
Lebesgue measure). 
\end{theorem}

By adapting D.~Sullivan's arguments, we obtain the following

\begin{thmain}\label{Conformal} 
Let $G$ be a finitely generated group of $C^2$ circle diffeomorphisms. 
Then any conformal measure of exponent $\delta\le 1$ without atomic part 
and which is supported on an infinite minimal invariant set is conservative.
\end{thmain}

The ergodicity (and uniqueness) of a conformal measure on a minimal set~$\Lambda$ 
was proved by D.~Sullivan when the group $G$ has the \emph{expansion property}, 
\emph{i.e.}~$G$ has no finite orbits, and $\NE(G)$ does not intersect~$\Lambda$:
\begin{theorem} [D.~Sullivan~\cite{Sullivan1}] 
Let $G$ be a group of $C^2$ circle diffeomorphisms acting with 
the expansion property. Then there is a unique conformal measure supported on the 
minimal set.
If the action is minimal, this is the Lebesgue measure. 
If there is an exceptional minimal set~$\Lambda$, this is the corresponding (normalized)
Hausdorff measure, which is then non-vanishing and finite. In the latter case, the exponent 
equals to the Hausdorff dimension of~$\Lambda$, which verifies $0<HD(\Lambda)<1$. 
\end{theorem} 

It is unclear whether in the general case the Hausdorff measure on the minimal set  
is finite and nonzero. (If this is the case, it would be a conformal measure.)
However, using the control of distortion technique for the expansion,
one can obtain the following uniqueness result for the conformal measure:

\begin{thmain}\label{ConformalStar}
Let $G$ be a finitely generated group of $C^2$ circle  
diffeomorphisms. 
\begin{enumerate}
\item If~$G$ acts minimally and satisfies property \smin, then 
the Lebesgue measure is the unique conformal measure which 
does not charge the orbit $G(\NE)$ of the 
set~$\NE$ of non-expandable points. Moreover, all the other (atomic) conformal measures 
(if any) are supported on $G(\NE)$, and their conformal exponents are 
strictly greater than~$1$.
\item If the action of $G$ carries an exceptional minimal set $\Lambda$ for which  
the property~\Ls\ holds, then there exists at most one conformal measure supported 
on~$\Lambda$ and not charging the $G$-orbit of $\Lambda\bigcap \NE$. If such a 
measure exists, then its exponent belongs to the interval~$(0,1)$.
\item In particular, in both the minimal and the exceptional minimal cases, the non-atomic 
conformal measures supported on the minimal invariant set (the whole circle 
or~$\Lambda$, respectively) are ergodic.
\end{enumerate}
\end{thmain}

As a final remark let us point out that, quite surprisingly, there exist examples of 
conformal measures on the circle whose conformal exponent exceeds one (that is, 
the dimension of the circle). These examples illustrate the restrictions 
imposed in Theorem~\ref{ConformalStar}:

\begin{example}\label{ex:GS}
There exists a Ghys-Sergiescu's non-minimal $C^2$ action on the circle 
of the Thompson group~$T$ such that, for every $\delta\ge 1$, there exists a
conformal measure of exponent~$\delta$ concentrated on the endpoints 
of the complementary intervals of the minimal set. 
\end{example}
\begin{example}\label{ex:PSLmin}
For the (minimal) standard $\PSL_2(\bbZ)$-action, for every exponent $\delta>1$ 
there exists a $\delta$-conformal measure concentrated on the orbit of the 
non-expandable point~$(0:1)$.
\end{example}

\subsection{Stationary measures}

Another concept related to the study of group actions is that of random dynamics. Namely, if 
in addition to an action of a group~$G$ on a compact space~$X$ is given a measure $m$ 
on $G$, then one can consider the left random walk on~$G$ generated by $m$, and the 
corresponding random process on~$X$. In other words,
one deals with \emph{random} sequences of compositions
$$
id, \, g_1, \, g_2 \circ g_1,\dots, \,g_n \circ g_{n-1} \circ \dots \circ g_1,\, \dots 
$$ 
where all the $g_j\in G$ are chosen independently and are distributed 
with respect to~$m$. The images $x_k=g_k\cdot\dots\cdot g_1 (x_0)$ of 
a given initial point $x_0\in X$ can be then considered as its ``random iterates'', 
as one has $x_{k+1}=g_{k+1}(x_k)$. Associated to this concept there is the following 
\begin{definition}
A measure~$\nu$ on $X$ is \emph{stationary with respect to $m$} if it coincides with the average of its images, that is, for every Borel $B\subset X$,
$$
\nu(B)=\int_G (g_*\nu)(B) \, dm(g).
$$
\end{definition}

This widely studied notion is in some sense analogous to that of an invariant measure for 
single maps. For instance, the existence of a stationary measure may be deduced by the 
classical Krylov-Bogolubov procedure of time averaging; moreover, an analogue of 
the Birkhoff Ergodic Theorem holds\dots
We recall more details on this in Section~\ref{s:background}. 

If for an action on the circle the first $\Diff^1$-moment\footnote{Formally 
speaking, it would be more exact to call it the 
first $\log$-$\Diff^1$-moment, as we are taking the logarithm of 
the $\Diff^1$-norm, in order to deal with a composition-subadditive
number. However, we prefer shortening the terminology.} 
$$
\int_G \log^+ \|g\|_{\Diff^1(S^1)} \, dm(g)
$$
of the measure~$m$ is finite, one can also define 
a \emph{random Lyapunov exponent} with respect to an ergodic stationary measure. 
If the support $\supp(m)$ does not preserve any 
measure on the circle, a result of 
P.~Baxendale~\cite{Baxendale} ensures the existence of an ergodic stationary 
measure whose random Lyapunov exponent is negative (which corresponds to 
some kind of a local contraction by random compositions).

Using the negativity of the random Lyapunov exponent, and somehow 
reversing the time of the dynamics, we obtain the following result relating 
the stationary measures to the Lyapunov expansion exponent of individual points. 

\begin{thmain}\label{thm:positive}
Let $G$ be a group of $C^1$ circle diffeomorphisms, and let $m$ be a 
measure on~$G$ having finite first word-moment. Assume that there is no 
measure on the circle which is invariant by all the elements in~$\supp(m)$. Then there exists a $m$-stationary measure~$\nu$ such that the Lyapunov expansion exponent is positive at $\nu$-almost every point.
\end{thmain}

This result turns out to be interesting for the study of the question of 
the regularity of the stationary measure that we explain below.

Namely, quite often the stationary measure turns out to be unique. More precisely, 
the local contraction, coming from the negativity of the random Lyapunov exponent 
of some stationary measure~$\nu$, enables to prove the uniqueness of the stationary 
measure in the basin of attraction of~$\nu$ (see~\cite{Antonov,DKN}). This provides 
the uniqueness of the stationary measure in the case where, in addition to the 
absence of a $\supp(m)$-invariant measure, either the support~$\supp(m)$ 
generates the whole group~$G$, or the action on the circle of the semigroup 
generated by~$\supp(m)$ is minimal.
We recall the precise definitions and statements of these results later 
in Section~\ref{ss:Random}.

Now, in the case of uniqueness of the stationary measure, the latter is 
either absolutely continuous, or singular with respect to the Lebesgue 
measure (as both the absolutely continuous and 
the singular parts would be stationary). 
This dichotomy is at the origin of very interesting problems; in particular, 
there is the following conjecture, that we learned from 
Y.~Guivarc'h, V.~Kaimanovich, and F.~Ledrappier:
\begin{conjecture}\label{conj3}
For any finitely supported measure $m$ on a lattice $\Gamma<\PSL_2(\bbR)$ 
whose support generates~$\Gamma$, the corresponding stationary measure 
on the circle is singular. 
\end{conjecture} 

This was proven by Y.~Guivarc'h and Y.~Le~Jan in~\cite{GLJ} for non-cocompact 
lattices (\emph{i.e.}, the quotient of the hyperbolic disc by the action has at least 
one cusp). Moreover, their result still holds if the measure $m$ on $\Gamma$ has 
\emph{finite first word-moment}, that is, if 
$$
\int_G \|g\|_{\mcF} \, dm(g)<\infty,
$$
where $\|\cdot\|_{\mcF}$ denotes the word-norm with respect to some finite generating 
set~$\mcF\subset \Gamma$.

In this direction our Theorem~\ref{thm:positive} provides the following

\begin{corollary}\label{c:Sing}
Assume that the Lyapunov expansion exponent for a finitely generated group~$G$ of 
$C^2$ circle diffeomorphisms is equal to zero. Then for any measure $m$ on~$G$ such 
that the corresponding stationary measure~$\nu$ is unique and non 
$\supp(m)$-invariant, $\nu$ is singular with respect to the Lebesgue one.

In particular, this holds for measures~$m$ with 
finite first word-moment, without $\supp(m)$-common invariant measures, 
and such that either the support $\supp(m)$ generates~$G$, 
or the semigroup generated by this support acts minimally on the circle. 
\end{corollary}

\begin{remark}
Corollary~\ref{c:Sing} applies (under the same conditions on the measure~$m$) 
for the minimal actions of the 
Thompson group~$T$ and of $\PSL_2(\bbZ)$ that we mentioned earlier. 
\end{remark}

\begin{remark}\label{r:limit}
According to the definition, the Lyapunov expansion exponent corresponds to an upper limit. Therefore, the ``exponentially expanding'' compositions for $\nu$-almost every point in the conclusion of Theorem~\ref{thm:positive} are proved to exist only for an infinite subsequence of lengths~$n_k$. However, under some more restrictive assumptions on the moments (that are satisfied, for instance, if the measure $m$ is finitely supported), one can prove that for $\nu$-almost every point the compositions with exponentially big derivative exist for \emph{every}~$n$.
\end{remark}

To end this paragraph, we would like to notice a subtle and interesting difference, concerning the question of the finiteness of the first moment; this difference was pointed out to us by V.~Kaimanovich. Namely, the first moment for a measure on a lattice $\Gamma<\PSL_2(\bbR)$ can be measured in two different ways: in the sense of (the logarithm of) the $\Diff^1$-norm\footnote{In fact, the value of $\log \|g\|_{\Diff^1}$ is equivalent to that of $\Var_{S^1} (\log g')$. A direct computation shows for $g\in\PSL_2(\bbR)$ that $\Var_{S^1} (\log g')=4 \dist_{hyp}(g(0),0)$, where the circle $S^1\subset\bbC$ is the unit circle, and the map~$g$ is naturally extended to the interior of the hyperbolic disc~$\bbD\ni 0$. On the other hand, $\PSL_2(\bbR)$ can be thought as the unit tangent bundle $T_1(\bbD)$  (which is hence a hyperbolic metric space), so the latter value is equivalent to the hyperbolic distance from~$g$ to the identity.}, and in the sense of the word-norm. Since 
$$
\log^+\|f\circ g\|_{\Diff^1} \le \log^+\|f\|_{\Diff^1} + \log^+\|g\|_{\Diff^1},
$$
finiteness of the first word-moment implies the finiteness of the first $\Diff^1$-moment. However, the converse does not hold. Indeed, the Furstenberg discretization procedure~\cite{Furstenberg2} allows to find a measure $m$ on $PLS_2(\bbZ)$ such that the $m$-stationary measure is the Lebesgue one. Due to the result of Y.~Guivarc'h and Y.~Le Jan~\cite{GLJ} (or due to our Corollary~\ref{c:Sing}), such a measure $m$ cannot be of finite first word-moment. Nevertheless (see~\cite{Kaimanovich-LePrince} and references therein), it can be chosen with a finite $\Diff^1$-moment! 

\subsection{Structure of the paper}

In Section~\ref{s:background}, we give some background on the problems that we consider 
and we recall several facts that will be used later. Section~\ref{s:questions} is devoted to 
the open questions. Some of these questions were widely known before this work, 
and some other naturally appeared when writing this paper. 
Section~\ref{s:distortion} is devoted to the description of one of the main tools in 
one-dimensional dynamics, namely the control of distortion technique. We also give therein
the proof of Theorem~\ref{Conformal}. In Section~\ref{s:examples}, we study 
two examples discussed before, the $\PSL_2(\bbZ)$ and Thompson group actions, 
and prove Theorems~\ref{Thompson} and~\ref{PSL2Z}. Finally, Section~\ref{s:proof} is 
devoted to the remaining proofs, \emph{i.e.} those of Theorems~\ref{description}, \ref{ConformalStar}, and~\ref{thm:positive}.


\section{Background}\label{s:background}

\subsection{Minimality, ergodicity, and exceptional minimal sets}\label{ss:minimality}

We begin by pointing out that all the assumptions of Conjecture~\ref{conj} are 
essential, and none of them can be omitted. Concerning the dimension, a celebrated
construction by H.~Furstenberg~\cite{Furstenberg} leads to examples of
area preserving diffeomorphisms of the torus~$\mathbb{T}^2$ which 
are minimal but not ergodic. These
diffeomorphisms are skew-products over irrational rotations, that is, maps of the form
$F:(x,y)\mapsto (x+\alpha, y+\varphi(x))$. Assume that the angle $\alpha$ is Liouvillian,
and that the cocycle corresponding to the function~$\varphi(x)$ is measurably trivial but
nontrivial in the continuous category. Then the map $F$ appears to be measurably conjugate
to an horizontal rotation $(x,y)\mapsto (x+\alpha, y)$, and thus non-ergodic; however,
the absence of a continuous conjugacy allows to show that it is minimal.

For the remaining hypotheses, following the general approach of \cite{DKN},
it is convenient to distinguish two cases according to the existence
or nonexistence of an invariant 
measure.

The hypothesis concerning smoothness is very subtle. Indeed, it is not difficult to construct
minimal circle homeomorphisms that are non-ergodic. However, the construction of $C^1$
diffeomorphisms with these properties is quite technical and much more difficult: see
\cite{OR}. (It is very plausible that, by refining the methods from \cite{OR}, one may
actually provide examples of $C^{1+\alpha}$ such diffeomorphisms for any $0 <\alpha <1$.)
For a non measure preserving example, one may follow (an easy extension of) the construction
in \cite{Ghys-Sergiescu} starting with a slight modification of the expanding map constructed
in \cite{Quas} (so that it becomes tangent to the identity at the end points). For $n \geq 10$,
this provides us examples of $C^1$ actions of the $n$-adic Thompson groups (which are finitely
presented) that are minimal but not ergodic. (We point out however that these examples 
seem to be non $C^{1+\alpha}$ smoothable for any $\alpha > 0$.)

Finally, to illustrate the finite generation hypothesis, one may
construct an example of an Abelian group action via a sequence of
actions of $G_n=\bbZ/2^n\bbZ$, where the $n^{\mathrm{th}}$ action is
obtained from the previous one by specifying a particular choice of
a ``square root'' of the generator. Such a choice is equivalent to
the choice of a two-fold covering $S^1/G_n \to S^1/G_{n+1}$. It may
be checked that with a well chosen sequence of actions, one can
ensure that the resulting action of the group $G=\bigcup_n G_n$ is
minimal, but it is non-ergodic and even non-conservative: there is a set of positive 
measure which is disjoint from all of its images by nontrivial elements of~$G$. Although 
this construction seems to be well known, we didn't find it in 
the literature, and for the reader's convenience we provide more 
details at the end of Section~\ref{s:distortion}. We point out, however, 
that there exists a simpler example (due to D.~Sullivan~\cite{Sullivan}) 
of a non-ergodic, minimal, smooth group action on the circle without 
invariant 
measure. Namely, fixing a Cantor set $\Lambda \subset S^1$ 
of positive Lebesgue measure, for each connected component $I$ of 
$S^1 \setminus \Lambda$ one chooses an hyperbolic reflection~$g_I$ 
with respect to the geodesic joining the endpoints of~$I$. Then the 
action of the group generated by the $g_I$'s is minimal (this can be 
checked using the fact that every orbit intersects all the complementary 
intervals of $\Lambda$, and thus accumulates everywhere on~$\Lambda$). 
Nevertheless, it is non-ergodic (and even non-conservative), since the 
set~$\Lambda$ of positive measure is disjoint from all of its nontrivial images.

Let us now consider the case of a subgroup~$G\subset \Diff^2(S^1)$ 
satisfying the hypotheses of Conjecture~\ref{conj}. We first point out 
that the ergodicity is a nontrivial issue even when~$G$ is generated by 
a single diffeomorphism. Indeed, Poincar\'e's theorem 
implies that every minimal circle homeomorphism is topologically 
conjugate to an irrational rotation. However, for ``an essential part'' of the 
set of minimal diffeomorphisms the conjugating map appears to be
singular, and therefore the ergodicity with respect to the Lebesgue measure after
conjugacy does not imply the ergodicity with respect to the Lebesgue measure before it.
Nevertheless, the conjecture for this case has been settled independently by
A.~Katok for $C^{1+bv}$ diffeomorphisms (see for instance~\cite{Katok}) and by
M.~Herman for $C^{1 + lip}$ diffeomorphisms (see~\cite{Herman}). Actually, Katok's
proof uses arguments of control of distortion for the iterations that are based on
the existence of decompositions of the circle into arcs
which are almost permuted by the dynamics (and which come from the good 
rational approximations of the rotation number).

If~$G$ has no element with irrational rotation number, the result above cannot 
be applied. However, in this case the dynamics has some hyperbolic 
behaviour. To show the ergodicity one then would like to apply the 
exponential expansion strategy. This classical procedure consists on
expanding very small intervals which concentrate a good proportion of
some invariant set, in such a way that the distortion (see a precise definition in Section~\ref{s:distortion})
of the compositions remains controlled. More precisely, the scheme works as follows. Let $A\subset S^1$
be an invariant measurable subset of positive Lebesgue measure. By Lebesgue's theorem, almost every
point~$x \!\in\! A$ is a density (or Lebesgue) point, that is, 
$$
\frac{\mu_L (U_{\delta}(x)\cap A)}{\mu_L (U_{\delta}(x))} \to 1 \quad \text{ as } \quad \delta \to 0,
$$
where $U_{\delta}(x)$ denotes the $\delta$-neighborhood of $x$. Now take $\delta > 0$
such that the proportion of points of~$A$ in $U_{\delta}(x)$ is very close to~$1$.
If one can expand this interval keeping a uniform bound for the distortion, then each one of the
``expanded'' intervals also has a proportion of points in~$A$ very close to~$1$. On the other hand, since
their length stay bounded away from zero, after passing to the limit along a sequence $\delta_n \to 0$
what we see is an interval in which the points in~$A$ form a subset of full relative measure. If the
action is minimal, this implies that $A$ is a subset of full measure in the circle.

The arguments that we have just cited were used by the third author in~\cite{Navas-ens} 
as well as by S.~Hurder in~\cite{Hurder} for establishing their ergodicity results that we mentioned 
in the Introduction: roughly speaking, if the expansion can be done 
``sufficiently quickly'', then minimal actions are necessarily ergodic. 

For exceptional minimal sets, the zero Lebesgue measure conjecture was proven by 
the third author for the case where for each $x \in \Lambda$ there exists $g \in G$ such that
$g'(x) > 1$ (see~\cite{Navas-ens}). Later on, S.~Hurder showed in~\cite{Hurder} that 
the Lebesgue measure of the
intersection $\Lambda \cap (S^1 \setminus \{x\mid 
\lambda_{exp}(x)=0\})$ is equal to zero.
Also, the conjecture has been proved by Cantwell and Conlon in~\cite{CC} for the 
case where the dynamics is~``Markovian''. 

Once again, both hypotheses of G.~Hector Conjecture~\ref{conj2} are essential, and one 
can construct counter-examples in the case where they are not 
satisfied (see for instance~\cite{Bowen} and~\cite{Herman} 
for the hypothesis concerning smoothness).

\subsection{Random dynamics}\label{ss:Random}

Random dynamical systems have been studied for a long time, and we are 
certainly unable to recall here all (and even the main) achievements of this theory. 
We shall then restrict ourselves to those that will be necessary for 
the exposition.

First, we would like to recall that a random dynamics can be modeled in terms of a single 
map. Indeed, consider the map 
$$
F: X\times G^{\bbN} \to X\times G^{\bbN}, \quad F(x,(g_i)_{i=1}^{\infty}) = (g_1(x), (g_{i+1})_{i=1}^{\infty}),
$$
which is a skew-product over the left shift on $G^{\bbN}$.
In terms of this map, instead of saying that we consider random compositions of maps, we can say that we take a random point in $G^{\bbN}$, distributed with respect to~$m^{\bbN}$, and then we consider the iterations of $F$ on the fiber over this point. A direct computation then shows that a measure $\nu$ is $m$-stationary if and only if the measure $\nu\times m^{\bbN}$ is $F$-invariant.

The latter remark allows to apply to the random dynamics all the arsenal of techniques 
from Ergodic Theory~--- Krylov-Bogolubov procedure (implying the existence of stationary 
measures), Birkhoff Ergodic Theorem (ensuring the convergence of random time averages), etc.

In particular, one can define Lyapunov exponents for a smooth random dynamics on a compact manifold, provided that the first $\Diff^1$-moment of $m$ is finite. We will not do this in a general situation, and we will restrict ourselves to the case of the dynamics on the circle. 
In this case, the \emph{random Lyapunov exponent} corresponding to a point~$x\in S^1$ and 
to a sequence~$(g_i)\in G^{\bbN}$ is defined as the limit
\begin{equation}\label{eq:random}
\lim_{n\to\infty} \frac{1}{n} \log (g_n\circ \dots g_1)'(x).
\end{equation}
Simple arguments show that, for a given measure~$m$ on~$G$ with 
finite first $\Diff^1$-moment, and for any $m$-stationary ergodic measure~$\nu$, 
the limit~\eqref{eq:random} is constant (in particular, it exists) almost everywhere 
w.r.t. the measure $\nu \times m^{\bbN}$. By Birkhoff Ergodic Theorem, this constant equals 
$$
\int_{S^1} \int_G \log g'(x) \, dm(g) \, d\nu(x),
$$ 
and we denote it by~$\lambda_{RD}(m;\nu)$. 

Now, according to a general principle in one-dimensional dynamics, which has been 
developed in the work of many authors, for a random dynamics on the circle that does 
not preserve any measure, ``random compositions contract''. In other words, under 
certain general assumptions (for instance, the system should be supposed to be non-factorizable) one can conclude that a long composition, most probably, will map almost all the circle (except for a small interval) into a small interval. Equivalently, for any two given points of the circle, most probably their orbits along the same random sequence of compositions will approach each other. 

We would like to recall here the following results illustrating this principle. In his 
seminal work~\cite{F63}, H.~Furstenberg proved the contraction statement for 
projective dynamics in arbitrary dimension (in particular, on the circle). 
The work of V.A.~Antonov~\cite{Antonov} established the contraction for any minimal and 
inverse-minimal non-factorizable random dynamics on the circle (one can also find an 
exposition of this work in~\cite{Kleptsyn-Nalski, Navas-Es}). In his excellent work, 
P.~Baxendale~\cite{Baxendale} studied the sum of the Lyapunov exponents for a smooth random
dynamics on a compact manifold of any dimension.


More precisely, P.~Baxendale proved that for such a dynamics, if the 
first $\Diff^1$-moment is finite (so that the random Lyapunov exponents 
are well-defined) and there is common invariant measure, 
there exists an ergodic stationary measure such that the sum of its Lyapunov exponents 
(which can be thought of as the exponential rate of volume changement) is negative. In particular, 
for the circle (as it is one-dimensional), this impies the negativity of the Lyapunov exponent, which, in its turn (due to the distortion control arguments) implies local contraction by the random dynamics.

Together with the results of V.A.~Antonov, the above result becomes 
a powerful tool for studying group dynamics on the circle. In particular, this was exploited 
by the authors in~\cite{DKN}, where they proved the (global) contraction 
property for a symmetric measure~$m$ (in fact, the same arguments 
work if the support~$\supp(m)$ 
generates the acting group as a semigroup). 

The global contraction property implies the uniqueness of the stationary measure (see~\cite{Antonov, DKN}). Moreover, if the contraction property holds only locally, then the stationary measure is still unique provided that the system is minimal. 




\section{Open questions}\label{s:questions}

We must point out that the actions of Thompson group~$T$ and of $\PSL_2 (\bbZ)$ that we
deal with in this article are (up to some easy modifications) the only minimal, 
smooth actions of non-Abelian groups on the circle for which we know that $\NE \neq \emptyset$. 
This motivates the following

\begin{question}\label{q:Star}
Does every (sufficiently smooth) minimal action on the circle of a 
non-Abelian finitely generated group satisfy property~\smin?
\end{question}

Both the positive or negative answer to this question would be interesting: the positive one
would lead to an interesting general property of minimal actions on the circle, and the negative
one would give an example of a ``monster'', certainly having very strange properties.

According to Theorems~\ref{Thompson} and~\ref{PSL2Z}, for the actions of~$T$ 
and $\PSL_2 (\bbZ)$ the corresponding Lyapunov expansion exponents are zero. 
Therefore, the following question makes sense:

\begin{question}\label{q:L0}
Let $G$ be a finitely generated non-Abelian group (perhaps having property~\smin) 
of (sufficiently smooth) circle diffeomorphisms. If the action of $G$ is minimal 
and $\NE(G)\neq \emptyset$, is it necessarily true that $\lambda_{\exp}(G)=0$~?
\end{question}

Once again, both the positive or negative answer to this question 
are interesting, the positive one leading to a general property, and 
the negative one providing us of an interesting and perhaps strange action.

A particular case of Question~\ref{q:L0}, closely related to 
Conjecture~\ref{conj3}, is the following one:
\begin{question}\label{q:Gamma}
Is it true that for every non-cocompact lattice $\Gamma<\PSL_2(\bbR)$, the Lyapunov expansion exponent of its action on the circle is zero?
\end{question}
A positive answer to this question looks very plausible. By joining it to our 
Theorem~\ref{thm:positive}, it would give another proof to the singularity 
theorem by Y.~Guivarc'h and Y.~Le Jan cited in the Introduction.

In fact, adapting the arguments of the proof of Theorem~\ref{PSL2Z} (and using some 
techniques from Riemann Surfaces Theory), one can show that, for every lattice as 
above, its action on the circle satisfies the property~\smin. Moreover, the $\NE$-set turns 
out to be non-empty and corresponds in some precise sense to the set of cusps in the 
quotient surface. The proofs of these facts are, however, rather technical, and we do not 
give them here since this would overload the paper.


For the study of conformal measures, the results we stated in the Introduction lead to many other questions that seem interesting to us. First, the fact that for a minimal dynamics the only non-atomic conformal measure is the Lebesgue one, was proven only under the assumption of property~\smin. In would be interesting to answer this question in general:
\begin{question}\label{q:conformal}
Is it true that for any minimal smooth action of a finitely generated group on the circle the only non-atomic conformal measure is the Lebesgue one?
\end{question}

By Theorem~\ref{ConformalStar}, a positive answer to Question~\ref{q:Star} would automatically imply a positive answer to Question~\ref{q:conformal},
but certainly the latter question can be attacked independently (and perhaps will 
be simpler to handle via some other way).

Analogous questions, as well as several new ones, are interesting in presence of an exceptional minimal set: does every finitely generated action with an exceptional minimal set satisfy property~\Ls? Is it true that for every finitely generated group action (not necessarily satisfying property~\Ls) with an exceptional minimal set~$\Lambda$, there exists at most one non-atomic conformal measure supported on it? Does such a measure always (or under the assumption of property~\Ls) exist? If yes, does it coincide with the normalized Hausdorff measure (which then will be non-vanishing and finite)? In the case of existence of such a measure, does its conformal exponent coincide with the the Hausdorff dimension of the minimal set? Is it true that, in the general (\emph{i.e.}, minimal dynamics 
or exceptional minimal set) situation, a conformal measure with exponent greater than one 
is atomic? 

To conclude this section, we would like to state a question due to \'E.~Ghys concerning 
the dichotomy between absolute continuity and singularity for stationary measures. 
To motivate this question, first notice that, in the examples of singular stationary measure
for minimal actions that we have already mentioned 
(Thompson group~$T$, non-cocompact lattices in~$\PSL_2(\bbR)$), 
the corresponding groups are generated by maps that are ``far'' from the identity. 

Moreover, singular stationary measures naturally appear for (expanding) actions 
of fundamental groups of closed genus $g>1$ surfaces. Indeed,
to each conformal structure on such a surface corresponds an action of 
its fundamental group on the circle (viewed as the boundary of its
universal cover, \emph{i.e.} the Poincar\'e disc). Through the actions corresponding 
to different complex structures are topologically conjugate, the conjugating 
map is always singular. Thus, among the stationary measures corresponding to 
different structures (for the same probability distribution on the group), at most one is 
absolutely continuous. Notice however that, once again, these groups are generated 
by maps that are ``far'' from the identity. 

In another direction, a result due to J.~Rebelo~\cite{Rebelo} asserts that topological 
conjugacies between non-solvable groups of circle diffeomorphisms generated by elements 
\emph{near the identity} are absolutely continuous. Thus, the above 
methods for obtaining a singular stationary measure stop working if we restrict 
ourselves to such actions. 


Due to all of this, it is interesting to find out if there are examples of singular 
stationary measures for actions generated by maps close to the identity:
\begin{question}[\'E.~Ghys]
Let $G$ be a non-Abelian group of $C^2$ circle diffeomorphisms without finite orbits and generated 
as a semigroup by finitely many elements close to rotations (in the sense of~{\rm \cite{Navas-ens}}). 
If $m$ is any 
measure supported on this system of generators, is it necessarily true that the associated stationary measure on~$S^1$ is equivalent to the Lebesgue measure? Is this true under the additional assumption that the set of generators and the distribution~$m$ are symmetric with respect to inversion?
\end{question}

It is interesting to notice that, under some assumptions, for analytic perturbations 
of the trivial system the equation for the density of the stationary measure admits 
at least a formal solution as a power series in the parameter.


\section{Control of distortion and conservativity}\label{s:distortion}

We begin this section by recalling several lemmas concerning control of distortion 
which are classical in the context of smooth one-dimensional dynamics. A more 
detailed discussion may be found, for instance, in~\cite{DKN}, and in many of the 
references therein. Therefore, we will not enter into the technical
details of their proofs here, and we will just briefly describe the ideas. We begin with a
general definition.

\begin{definition} Given two intervals $I,J$ and a $C^1$ map $F\!: I\to J$ which is a
diffeomorphism onto its image, we define the \emph{distortion coefficient} of $F$ on~$I$ by
$$
\varkappa (F;I) :=\log \Big( \frac{\max_I F'}{\min_I F'} \Big),
$$
and its \emph{distortion norm} by
$$
\eta (F;I) := \sup_{\{x,y\} \subset I} \frac{\log \big(
\frac{F'(x)}{F'(y)}\big)}{|F(x)-F(y)|}
= \max_{J} \left| \big( \log ( (F^{-1})')\big)'\right|.
$$
\end{definition}

It is easy to check that the distortion coefficient is subadditive under composition. Moreover,
by Lagrange Theorem, one has $\varkappa(F,I)\le C_F |I|,$ where the constant $C_F$ depends only
on the $\Diff^2$-norm of~$F$ (indeed, one can take $C_F$ as being the maximum of the absolute
value of the derivative of the function $\log(F')$). This implies immediately the following

\begin{proposition}\label{p:sum}
Let $\mathcal{F}$ be a subset of $\Diff^2_+(S^1)$ which is bounded with respect to the $\Diff^2$-norm.
If $I$ is an interval on the circle and $f_1,\dots,f_n$ are finitely many elements chosen
from $\mcF$, then
$$
\varkappa(f_n \circ \cdots \circ f_1;I) \le
C_{\mcF} \sum_{i=0}^{n-1} |f_i \circ \cdots \circ f_1 (I)|,
$$
where the constant $C_{\mcF}$ depends only on the set~$\mcF$.
\end{proposition}

In other words, if we compose several ``relatively simple'' maps, then a bound
for the sum of the lengths of the successive images of an interval $I$ provides
a control for the distortion of the whole composition over $I$. Using this fact 
one can show the following

\begin{corollary}\label{cor:estimates}
Under the assumptions of Proposition~\ref{p:sum}, let us fix a point $x_0 \in I$,
and let us denote $F_i := f_i \circ \cdots \circ f_1$, $I_i := F_i(I)$, and
$x_i := F_i(x_0).$ Then the following inequalities hold:
\begin{equation}
\exp \big( -C_{\mcF}\sum_{j=0}^{i-1} |I_j| \big) \cdot \frac{|I_i|}{|I|} \le F_i'(x_0)
\le \exp \big( C_{\mcF}\sum_{j=0}^{i-1} |I_j| \big) \cdot \frac{|I_i|}{|I|} ,
\end{equation}
\begin{equation}\label{eq:lsum}
\sum_{i=0}^n |I_i| \le  |I|
\exp \big( C_{\mcF} \sum_{i=0}^{n-1} |I_i| \big) \sum_{i=0}^{n} F_i'(x_0).
\end{equation}
\end{corollary}

Notice that the sum in the exponential in~\eqref{eq:lsum} goes up to $i \! = \! n-1$,
while the sum of the lengths in the left hand side expression goes up to~$n$.
Using an induction argument, this seemingly innocuous remark appears to be
fundamental for establishing the following

\begin{proposition}\label{bound}
Under the assumptions of Proposition~\ref{p:sum}, given a point
$x_0\in S^1$ let us denote \, $S := \sum_{i=0}^{n-1}
F_i'(x_0)$. \, Then for every \, $\delta
\leq \log(2) / 2 C_{\mcF} S$ \, one has \,
$\varkappa(F_n,U_{\delta/2}(x_0)) \le 2 C_{\mcF} S \delta.$
\end{proposition}

As a consequence, if the sum of the derivatives is not too big, then up to a 
multiplicative constant one can approximate the length of the image interval 
in Proposition~\ref{p:sum} by the length of the initial interval $I$ times 
the derivative of the composition at a given point in $I$. This simple fact 
allows us already to prove the conservativity of conformal measures.
 
\begin{proof}[Proof of Theorem~\ref{Conformal}.] Let $\mathcal{F}$ be a finite family 
of generators of $G$ as a semi-group. 
Suppose that there exists a Borel subset $A$ of the circle such that 
$\mu(A)>0$, and $\mu(A \cap g(A))=0$ for every nontrivial
element $g \in G$. This immediately yields $\mu (g(A) \cap h(A)) = 0$
\hspace{0.015cm} for every $g \neq h$ in $G$, which gives
$$
1 \geq \mu \Big( \bigcup_{g \in G} g(A) \Big) =
\sum_{g \in G} \mu \big( g(A) \big) = \sum_{g \in G} \int_{A} g'(x)^{\delta}
\hspace{0.01cm} d\mu(x) = \int_{A}  \Big( \sum_{g \in G} g'(x)^{\delta} \Big)
d\mu(x).
$$ 
Therefore, for $\mu$-almost every point $x \!\in\! A$, the sum
$\sum_{g \in G} g'(x)^{\delta}$ converges, 
and since $\delta\le 1$ the same holds for the sum 
$S(x) := \sum_{g\in G} g'(x)$. Let us fix one of these
points~$x_0$, also belonging to the minimal set $\Lambda$ (we can do this, 
as the measure~$\mu$ is concentrated on $\Lambda$), and 
let $I$ be an open neighborhood centered at $x_0$
having length strictly smaller than $ \log(2) / 2
C_{\mcF} S(x_0)$. We claim that for every $g \!\in\! G$, 
and every $x\in I$, one has
$g'(x) \leq 2 g'(x_0)$. Indeed, this follows directly from
Proposition~\ref{bound} by writing $g$ as a product of generators.
Now the above implies that the $\mu$-measure of the set $B :=
\bigcup_{g \in G} g(I)$ is smaller than or equal to
$$
2^{\delta} \mu(I) \sum_{g \in G} g'(x_0)^{\delta}.
$$
Since $\mu$ is non-atomic, if $I$ is chosen small enough, then the value of this expression is 
strictly smaller than~$1$. If this is the case, the complementary set of~$B$ is of positive $\mu$-measure, and hence intersects~$\Lambda$. On the other hand, $B$ is an open $G$-invariant set containing $x_0\in\Lambda$. Therefore, $\Lambda\cap (S^{1}\setminus B)$ is a nonempty closed invariant set, strictly containted in~$\Lambda$, and this contradicts the minimality of~$\Lambda$.
\end{proof}

As we have seen in the Introduction, D.~Sullivan 
Theorem~\ref{thm:conservativity} is no longer true for 
non finitely generated groups of circle diffeomorphisms acting minimaly. 
For the sake of completeness, we provide below the details of the already 
mentioned example of a non finitely generated Abelian group of circle 
diffeomorphisms whose action is minimal but non-conservative. 

The construction works by induction. Fix a 
dense sequence of points $x_n$ in $S^1$. Let $g_1$
the Euclidean rotation of order~2, and assume that for an integer $n
\geq 2$ a generator $g_{n-1}$ of $G_{n-1}$ has been already
constructed. Let $p_{n-1} \!: S^1 \to S^1$ be the $(n-1)$-fold
covering map induced by $g_{n-1}$. For $\varepsilon_n > 0$ small
enough, the set $p_{n-1}^{-1} (p_{n-1} (U_{\varepsilon_n} (x_n))$ is
formed by $2^{n-1}$ disjoint intervals, and the lengths of these
intervals tend to zero as $\varepsilon_n$ goes to zero. Let us
enumerate these intervals (modulo $2^{n-1}$ and respecting their
cyclic order on $S^1$) by $I_{n-1}^1, \ldots, I_{n-1}^{2^{n-1}}$,
and let us denote by $J_{n-1}^i$ the maximal open interval to the
right of $I_{n-1}^i$ contained in the complementary set of the union
of the $I_{n-1}^j$'s. Now choose a generator $g_n$ of $G_n$ sending
each $I_{n-1}^i$ (resp. $J_{n-1}^i$) into $J_{n-1}^i$ (resp.
$I_{n-1}^{i+1}$), by appropriately lifting from the quotient $S^1/G_{n-1}$ a 
diffeomorphism that interchanges $p_{n-1}(U_{\varepsilon_n}(x_n))$ and 
its complementary. 

Notice that every $G_n$-orbit intersects the interval $U_{\varepsilon_n}(x_n)$. 
It is not difficult to deduce from this that, if the sequence $\varepsilon_n$ 
tends to zero as $n$ goes to infinity, then the action of $G  := \bigcup_n G_n$ 
is minimal. To ensure the non-conservativity we choose $\varepsilon_n$
sufficiently small so that
$$
\Leb \big( p_{n-1}^{-1} (p_{n-1} (U_{\varepsilon_n}(x_n))) \big) < \frac{1}{2^{n+1}},
$$
and we define a decreasing sequence of sets $A_n$, each of which is disjoint from
its nontrivial $G_n$-images, by
$$
A_0 := S^1, \quad
A_n := A_{n-1} \setminus p_{n-1}^{-1} \big( p_{n-1} (U_{\varepsilon_n}(x_n)) \big).
$$
By construction, the intersection $A := \bigcap_{n} A_n$ is a (measurable) set which is
disjoint from all of its images under nontrivial elements in $G$. Moreover, 
$$
\Leb(A) \ge 1-\sum_{n \ge 1} \Leb \big( p_{n-1}^{-1}(p_{n-1}(U_{\varepsilon_n}(x_0))) \big)
\ge 1 - \sum_{n \geq 1} \frac{1}{2^{n+1}} = 1 - \frac{1}{2}>0.
$$
This shows that the action is non-conservative.

\vspace{0.1cm}

To close this section we would like to point out that, to the best of our knowledge, 
the only examples of minimal non-ergodic group actions by $C^2$ circle diffeomorphisms 
that there exist in the literature are constructed by prescribing a positive measure 
set which is disjoint from all of its images ({\em i.e.}, they are actually non-conservative). 
This motivates the following

\begin{question}
Is every minimal and conservative action of a (non finitely 
generated) group by $C^2$ circle diffeomorphisms necessarily ergodic?
\end{question}

Notice that the minimal non-ergodic examples using Quas' construction that we 
mentioned in the Introduction are based on a different idea. However, these actions 
seem to be non $C^2$ smoothable (in many cases this follows from our Theorem~\ref{description}).


\section{Examples}\label{s:examples}
\subsection{The smooth, minimal action of Thompson group~$T$}\label{ss:Thompson}
Recall that Thompson group~T is the group of circle homeomorphisms which are piecewise linear
in such a way that all the break points, as well as their images, are dyadic rational numbers,
and which induce a bijection of the set of dyadic rationals (notice that these properties
force the derivatives on the linearity intervals to be integer powers of~$2$).

As \'E.~Ghys and V.~Sergiescu have cleverly noticed in \cite{Ghys-Sergiescu}, 
the dynamics of this group is somehow ``generated'' by a single (non invertible) 
map, namely $\varphi_0: x \mapsto 2x \mod 1$.
Indeed, $\varphi_0$ has a (unique) fixed point $x=0$ whose preorbit is exactly the set of
dyadic rationals, and Thompson group~$T$ is the set of homeomorphisms obtained by 
gluing finitely maps of the form $\varphi_0^{-k}\circ \varphi_0^l$ at some dyadic rationals
(here, for $\varphi_0^{-k}$ one can choose any of the corresponding $2^k$ branches).

The main argument of the construction in~\cite{Ghys-Sergiescu} consists in replacing $\varphi_0$
by another degree-two smooth monotonous map $\varphi$ fixing the point $x=0$ and being 
sufficiently tangent to the identity at this point. One can then define the set of
``$\varphi$-dyadically
rational'' points as the $\varphi$-preorbit of~$0$, and one can make correspond, to each
element $f \in T$, the map $[f]_{\varphi}$ which is obtained by gluing (in a coherent way)
the branches of $\varphi^{-k}\circ \varphi^l$ instead of $\varphi_0^{-k}\circ \varphi_0^l$ at the
corresponding $\varphi$-dyadically rational points. The issue here is that, since~$\varphi$ 
is tangent to the identity at~$0$, the maps obtained after gluing are smooth (actually, as smooth
as the order of the tangency is). Thus, $f\mapsto [f]_{\varphi}$ is a smooth action of the
Thompson group~$T$ on the circle.

By choosing appropriately the map~$\varphi$, the previous action can be made either minimal or having
a minimal invariant Cantor set. Here we are going to deal with the first case, which is ensured if
$\varphi$ satisfies $\varphi'(x)>1$ for all $x\neq 0$.

\begin{figure}[!h]
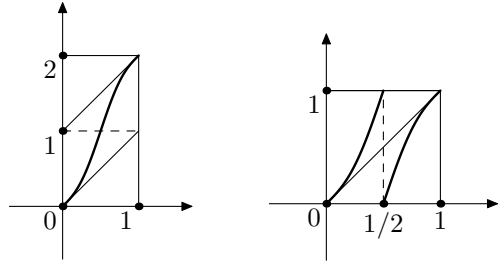

    \begin{center}
        \includegraphics{DKN1.1}
        \qquad
        \includegraphics{DKN1.2}
    \end{center}
    \caption{The map $\varphi$}
\end{figure}

We can now pass to the 

\begin{proof}[Proof of  Theorem~\ref{Thompson}] The first claim of the theorem, namely, the
equality $\NE=\{0\}$, is rather simple. Indeed, it is quite clear that for every point~$x\neq 0$
one can find an element~$g$ in the modified Thompson group~$[T]_{\varphi}$ that
coincides with $\varphi$ in a neighborhood of $x$, and this implies that $g'(x) = \varphi'(x)>1$.
On the other hand, every $g\in [T]_{\varphi}$ coincides in some right neighborhood of
$0$ with a map of the form $\varphi^{-k} \circ \varphi^l$ for some non negative integers $k,l$.
Therefore
$$g'(0)= (\varphi^{-k}\circ \varphi^l)' (0) =
(\varphi^{-k})'(\varphi^l(0))\cdot (\varphi^l)'(0) = (\varphi^{-k})'(0) \le 1,$$
where the third equality follows from the fact that $0$ is a neutral fixed point of~$\varphi$,
while the last inequality comes from the fact that $\varphi$ is a non-uniformly expanding map.

\begin{remark}\label{etienne} As it was pointed out to us by \'E.~Ghys, for slightly different maps
$\varphi$ the $\NE$-set may contain finitely many $\varphi$-periodic orbits along which the derivative
of $\varphi$ equals $1$. For instance, for the map $\varphi:x\mapsto 2x-\frac{1}{6\pi}\sin(6\pi x)$,
the induced action of $T$ (is minimal and) satisfies $\NE([T]_{\varphi})=\{0,1/3,2/3\}$.
\end{remark}

To prove the equality $\lambda_{\exp}([T]_{\varphi})=0$, we fix a finite set of elements
$\mcF=\{f_1,\ldots,f_s\}$ which generates $[T]_{\varphi}$
as a semigroup. Each of these elements $f_i$ coincide locally with maps of the form
$\varphi^{-k_{i,j}} \circ \varphi^{l_{i,j}}$. If we let $L:= \max_{i,j} \{ l_{i,j} \}$, 
then any composition of the generators having length~$n$ writes, 
near and to the right of a given point $x$, in the form
$$f_{i_1}\circ \dots \circ f_{i_n}|_{[x,x+\varepsilon]} =
\varphi^{-k_{j_1}} \circ \varphi^{l_{j_1}} \circ \dots \circ
\varphi^{-k_{j_n}} \circ \varphi^{l_{j_n}}|_{[x,x+\varepsilon]}.$$
Notice that none of the compositions $\varphi^{-1}\circ \varphi$ can be simplified. However, the identity
$\varphi\circ \varphi^{-1}=\id$ still holds, and this allows to reduce the above expression to
$$f_{i_1}\circ \dots \circ f_{i_n}|_{[x,x+\varepsilon]} =
\varphi^{-k}\circ \varphi^{l}|_{[x,x+\varepsilon]},$$
where $l\le L n$. Thus,
$$(f_{i_1}\circ \dots \circ f_{i_n})'(x) =
(\varphi^{-k}\circ \varphi^{l})'(x) \le (\varphi^l)'(x)\le (\varphi^{Ln})'(x),$$
where the inequalities follow from the non-uniform expansivity of~$\varphi$. Hence,
to show that the Lyapunov expansion exponent of $[T]_{\varphi}$ is zero, it suffices
to show that the same holds for the map~$\varphi$. To do this we will use the following
result due to T.~Inoue \cite{Inoue}, which will be discussed at the end of this section
since some of the involved ideas will be used latter.

\begin{lemma}[T.~Inoue~\cite{Inoue}]\label{l:varphi}
For Lebesgue-a.e. point $x \in S^1$, the \emph{time averages} measures
$$
\mu_{n,x}:=\frac{1}{n}\sum_{i=0}^{n-1} \delta_{\varphi^i(x)}
$$
converge to the Dirac measure~$\delta_0$ concentrated 
at the neutral fixed point~$0$ of~$\varphi$.
\end{lemma}

Using Lemma~\ref{l:varphi}, classical arguments from Ergodic Theory show that the
Lyapunov exponent of the map~$\varphi$ is a.e. equal to zero. Indeed, since
for every point $x \in S^1$ and every $n\in\mathbb{N}$ one has
$$\frac{1}{n} \log (\varphi^n)'(x) = \frac{\log \varphi'(x) +
\log \varphi' (\varphi(x)) + \dots + \log \varphi' (\varphi^{n-1}(x))}{n}
= \int_{S^1} \!\!\! \log \varphi'(s) \, d\mu_{n,x}(s),$$
and since for a.e. $x\in S^1$ one has $\mu_{n,x}\xrightarrow[n\to\infty]{*-weakly} \delta_0$,
one concludes that, for a.e. $x\in S^1$,
$$\frac{1}{n} \log (\varphi^n)'(x) =
\int_{S^1} \log \varphi'(s) \, d\mu_{n,x}(s) \xrightarrow[n\to\infty]{}
\int_{S^1} \log \varphi'(s) \,d\delta_0 = \log \varphi'(0)=0.$$
Therefore, the Lyapunov exponent of the map~$\varphi$ is a.e. equal to zero, and this implies that
the same holds for the action of $[T]_{\varphi}$, thus concluding the proof of Theorem~\ref{Thompson}.
\end{proof}

We now give the sketch of the proof of Lemma~\ref{l:varphi}
since the ideas will be very useful in the next section. First recall that, for
every uniformly expanding, smooth circle map, there exists an absolutely continuous
ergodic invariant 
measure whose density is strictly positive and away from
zero; moreover, the same holds for maps of the interval having infinitely many branches,
provided that there is a uniform bound for the distortion norm and the expansiveness of all
of the branches (see for instance \cite{mane}). However, the situation which is considered
in the lemma is slightly different: although there are only finitely many branches,
due to the presence of a parabolic fixed point the map in non-uniformly expanding.

Nevertheless, the neutral fixed point can be somehow ``removed'' in
the following way. For each point $c$, denote by $[c]_{\varphi}$ the
preimage of $c$ under the (topological) conjugacy between $\varphi$
and $\varphi_0$. Since $[1/3]_{\varphi}$ is a $\varphi$-periodic
point of period two, the interval $J:=[a,b]$, where
$a:=[1/3]_{\varphi}$ and $b:=\varphi(a)=[2/3]_{\varphi}$, is a
``fundamental domain'' for the expansion both on the left and 
on the right of the neutral fixed point. Indeed, the restriction of
$\varphi$ to $\left[0, a\right]$ (resp. to $\left[b,1\right]$) is
one to one and onto $\left[0, b\right]$  (resp. $\left[a,1\right]$):
see Figure 2.

\begin{figure}[!h]
    \begin{center}
        \includegraphics{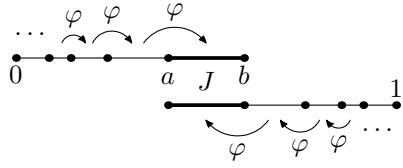}\label{DKN2}
    \end{center}
    \caption{A simultaneous fundamental domain}
\end{figure}

Consider the first-return map
$\Phi\!: J\to J$, as well as the return-time function $\tau\!:J\to\bbN$, which are given by
$$\Phi(x):=\varphi^{\tau(x)}(x), \quad \tau (x):=\min\{n\ge 1 \mid \varphi^n(x)\in J \}.$$

\begin{figure}[!h]
    \begin{center}
\includegraphics{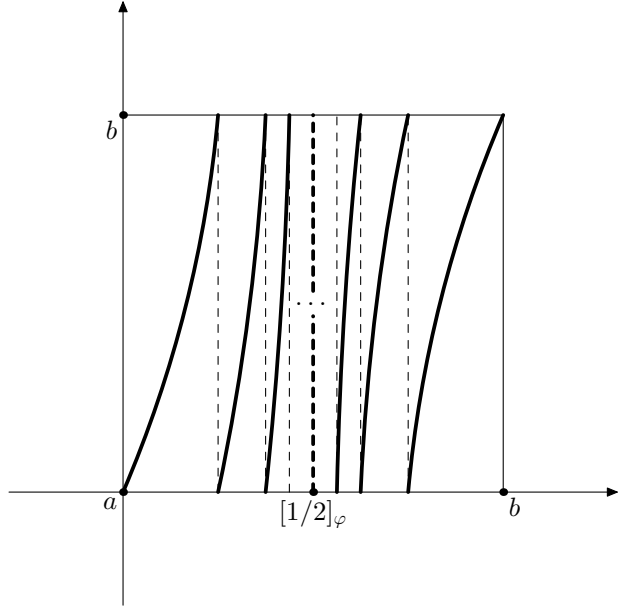}
    \end{center}
    \caption{The first-return map~$\Phi$}
\end{figure}

The map $\Phi$ is in fact an infinite-degree map with infinitely 
many discontinuity points. However, every maximal interval of
continuity $I$ of $\Phi$ is mapped onto~$J$, and the distortion 
and the expansiveness of all of the restrictions $\Phi_{I}$ are 
uniformly bounded. More precisely, since the images of $I$ under the 
maps~$\id,\varphi,\varphi^2,\dots,\varphi^{\tau(x)-1}$ are pairwise
disjoint, and hence the sum of their lengths does not exceed the
total length of the circle, the estimates from Section~\ref{s:distortion} 
provide a bound for the distortion norm of $\Phi_I$ which is independent 
of $I$; moreover, the new map $\Phi$ is strictly and uniformly expanding. 
(Compare Lemma \ref{l:fixed-exp}.) Together with what precedes, this allows 
to ensure the existence of an absolutely continuous ergodic invariant
measure~$\nu$ for $\Phi$ with a strictly positive density.

Consider now the sequence of iterates by the map~$\varphi$ of a Lebesgue generic point 
$x\in S^1$. Up to a finite number of initial steps, we can suppose that the point~$x$ 
belongs to the interval~$J$, and then its orbit can be divided into segments 
according to the arrivals to~$J$:
\begin{multline*}
x,\varphi(x),\dots,\varphi^{\tau(x)-1}(x); \Phi(x),\varphi(\Phi(x)), \dots,
\varphi^{\tau(\Phi(x))-1}(\Phi(x)); \dots ;
\\
\Phi^m(x),\varphi(\Phi^m(x)), \dots, \varphi^{\tau(\Phi^m(x))-1}(\Phi^m (x)); \dots
\end{multline*}
On the one hand, since the measure $\nu$ is absolutely continuous and has positive density, for
a Lebesgue generic point $x$ the sequence $x,\Phi(x),\Phi^2(x),\dots $ is distributed with
respect to~$\nu$. On the other hand, the return-time function $\tau$ has a non locally
integrable ``singularity'' (of type $1/x$) at the point~$x=[1/2]_{\varphi}$. Hence, due to
Birkhoff Ergodic Theorem, for a.e. $x\in J$ one has
$$\frac{\tau(x)+\tau(\Phi(x))+\dots+\tau(\Phi^{m-1}(x))}{m} 
\xrightarrow{\hspace{0.5cm}} +\infty \quad \text{ as  } m\to\infty,$$
and therefore
$$\frac{m}{\tau(x)+\tau(\Phi(x))+\dots+\tau(\Phi^{m-1}(x))} 
\xrightarrow{\hspace{0.5cm}} 0 \quad \text{ as  } m \to \infty.$$

Now for every fixed $\varepsilon > 0$ the points in $S^1 \setminus U_{\varepsilon}(0)$ fall into
$J$ in a bounded number of iterations. More precisely, there exists a constant $N=N_{\varepsilon}$
such that for every $x \notin U_{\eps}(0)$ one has $\varphi^j (x) \in J$ for some $j < N$. Hence,
for each $x \in J$, the time spent by a segment of
$\varphi$-orbit of length~$n$ outside $U_{\varepsilon}(0)$
is comparable to the number of returns to~$J$:
$$\# \{0 \le j \le n-1 \mid \varphi^j (x)\notin U_{\varepsilon}(0) \}
\le N \hspace{0.02cm} (m(n,x) + 1),$$
where
$$
m(n,x):=\max\{m\mid \tau(x) + \dots + \tau(\Phi^{m-1}(x)) \le n -1 \}.
$$
This implies that
\begin{equation}\label{eq:near-0}
\frac{\# \{0\le j \le n-1 \mid \varphi^j(x)\notin U_{\varepsilon}(0) \}}{n} \le
\frac{N \hspace{0.02cm} (m(n,x) + 1)}{\tau(x)+\tau(\Phi(x))+\dots + \tau(\Phi^{m(n,x)-1}(x))}
\xrightarrow[n\to\infty]{} 0.
\end{equation}
Thus, the proportion of time spent outside $U_{\varepsilon}(0)$ tends to~$0$ for a.e. $x\in J$,
and hence for a.e. $x\in S^1$. As $\varepsilon>0$ was arbitrary, (up to some technical details)
this concludes the proof of Lemma~\ref{l:varphi}.

\vspace{0.15cm} 

We close this section by giving an explicit construction for Example~\ref{ex:GS}.
Let us consider the Ghys-Sergiescu's non-minimal  
action of Thompson group $[T]_{\varphi}$, associated to 
a degree-two smooth circle map $\varphi$ 
with the following properties: 
\begin{itemize} 
\item It has exactly two fixed points $x_-$ and $x_+$. 
\item It is tangent to the identity at $x_-$ and $x_+$. 
\item Outside the invariant interval $[x_-,x_+]$, one has $\varphi'>1$.
\end{itemize} 
The latter property guarantees that when we shrink the components of the preimages 
of $I=(x_-,x_+)$ by the powers of $\varphi$, 
the induced map becomes topologically conjugate to~$\varphi_0$. 
This implies that the complement $\Lambda$ of the union of the preimages of $I$ 
is an exceptional minimal set for $[T]_{\varphi}$. 

For each preimage $y$ 
of $x_+$ by a power $\varphi^n$ of $\varphi$, let $I_y$ be the component of 
$\varphi^{-n}(I)$ containing the point $y$. 
By construction, all these intervals 
are disjoint. By the distortion arguments developped 
in Section~\ref{s:distortion}, 
there exists a constant $C>0$ depending only on~$\varphi$, such that $(\varphi^n )'(y) \geq \frac{C} {|I_y|}$.
Thus, the series 
\[ S= \sum _{(n,y) : \, \varphi^{n} (y) = x_+, \varphi^{n-1}(y)\neq x_+} \frac{1}{(\varphi^n)'(y)} \]
converges, and hence, for every $\delta\ge 1$, the value of the sum
\[ S_{\delta}:= \sum _{(n,y) : \, \varphi^{n} (y) = x_+, \varphi^{n-1}(y)\neq x_+} \frac{1}{[(\varphi^n)'(y)]^{\delta}} \]
is finite. Therefore, the 
measure 
\[\mu_{\delta} := \frac{1}{S_{\delta}} \sum _{n\in {\bf N},\ \varphi^{n} (y) = x_+\neq \varphi^{n-1}(y)} 
\frac{{\Dirac}_y}{[(\varphi^n)'(y)]^{\delta}} \]
is a $\delta$-conformal measure for $[T]_{\varphi}$
supported on the orbit of the point~$x_+$. 


\subsection{The case of $\PSL_2(\bbZ)$}\label{ss:PSL-2-Z}

To deal with the (canonical) action of $\PSL_2(\bbZ)$ on $S^1=\bbP(\bbR^2)$, we pass to an affine
chart on $\bbP(\bbR^2)=\bbR\cup\{\infty\}$ using the coordinate $\theta \mapsto \ctg(\theta)$.
Then the minimality follows easily from the density of $\bbQ$ in $\bbR$: every orbit accumulates
to the infinity ({\em i.e.}, the point $(1:0)$), and $G(\infty)=\bbQ \cup \{\infty\}$.


To show that the point $(1:0)$ is non-expandable first notice that, in the coordinate above,
an element $F=\left[\left(\begin{smallmatrix} a& b\\ c& d\end{smallmatrix}\right)\right]$
in $\PSL_2(\bbR)$ is given by
$$x\stackrel{\tilde{F}}{\mapsto} \frac{ax+b}{cx+d},$$
and thus its derivative at the point $x$ is
$$\widetilde{F}'(x)=\frac{ad-bc}{(cx+d)^2} = \frac{1}{(cx+d)^2}.$$
For $x=0$ this gives $\widetilde{F}'(0)=1/d^2$. Now, coming back to the original
coordinate~$\theta$, we have $\ctg'(\pi/2)=1$ and $\arcctg'(x)=1/(1+x^2)$;
therefore, if $d \neq 0$,
\begin{equation}\label{eq:derivative}
F'(0)=1\cdot\frac{1}{d^2}\cdot\frac{1}{1+(b/d)^2} =
\frac{1}{b^2+d^2}.
\end{equation}
By continuity, the same formula holds when $d = 0$. If $F$ belongs to
$\PSL_2(\bbZ)$ then $b,d$ are in $\bbZ$ and cannot be both equal to $0$.
Hence, the equality~\eqref{eq:derivative} shows that $F'(0) \leq 1$. A similar argument
shows that the point $(0:1)$ is also non-expandable.

By pursuing slightly the above computations, one easily checks that
the derivative of the map $F$ at a point $\theta\in S^1$ is equal to
\begin{equation}
F'(\theta)=\frac{\|(u,v)\|^2}{\|F(u,v)\|^2},
\label{deriv}
\end{equation}
where $(u,v)$ is any nonzero vector in the direction given by the angle~$\theta$.
This formula will strongly simplify the proof of the nullity of the Lyapunov expansion
exponent. For this, instead of working directly with $\PSL_2(\bbZ)$, we will work
with the subgroup $G_2$ which is the kernel of the natural map
$\PSL_2(\bbZ)\to SL_2(\bbZ/2\bbZ)$. Since $G_2$ is of finite index in
$\PSL_2(\bbZ)$, the corresponding actions have zero or positive
Lyapunov expansion exponents simultaneously.

It is well-known that $G_2$ is a
free group, and that one system of generators is given by
$f_1=\left[\left(\begin{smallmatrix} 1&2\\ 0&1\end{smallmatrix}\right)\right]$ and
$f_2=\left[\left(\begin{smallmatrix} 1&0\\ 2&1\end{smallmatrix}\right)\right]$.
One way to show this is by applying the Ping-Pong
Lemma (see {\em e.g}.~\cite{Ghys-actions}) to the sets
$$
A_+=\{\theta\in [0,\pi/4]\}, \quad B_+=\{\theta\in [\pi/4,\pi/2]\},
$$
$$
B_-=\{\theta\in [\pi/2,3\pi/4]\}, \quad A_-=\{\theta\in [3\pi/4,\pi]\}.
$$
(Notice that under the identification $S^1=\mathbb{P}(\bbR^2)$, the angle $\theta$
is measured modulo~$\pi$, and not modulo~$2\pi$, as usually.) Indeed, one has
$$f_1^{-1}(A_+) = A_+ \cup B_- \cup B_+, \quad f_1(A_-) = A_- \cup B_- \cup B_+,$$
$$f_2^{-1}(B_+) = B_+ \cup A_- \cup A_+, \quad f_2(B_-) = B_- \cup A_- \cup A_+.$$

We will denote by $\mcF = \{f_1,f_1^{-1},f_2,f_2^{-1}\}$ the finite set of elements
generating $G_2$ as a semigroup. Notice that for the action of (the representatives
of) these elements on a vector $(u,v)$, one has the following possibilities:
\begin{enumerate}
\item\label{i:generic} If $|u|\neq |v|$, $|u| \neq 0$, and $|v| \neq 0$,
then there is a unique element in $\mcF$ which decreases the norm
of~$(u,v)$, while the other generators strictly increase it.
\item\label{i:axis} If $|u|=0$, then $f_2^{\pm 1}$ preserve the norm of~$(u,v)$, while $f_1^{\pm 1}$ increase it.
\item[\ref{i:axis}')] If $|v|=0$, then $f_1^{\pm 1}$ preserve the norm of~$(u,v)$, while $f_2^{\pm 1}$ increase it.
\item\label{i:diagonals} If $u=v$, then $f_1^{-1}$ and $f_2^{-1}$ preserve the norm of~$(u,v)$, while
$f_1$ and $f_2$ increase it.
\item[\ref{i:diagonals}')] If $u=-v$, then $f_1$ and $f_2$ preserve
the norm of~$(u,v)$, while $f_1^{-1}$ and $f_2^{-1}$ increase it.
\end{enumerate}
Using~(\ref{deriv}), one may translate all of this to the original
action on the circle, thus showing that for any point~$\theta$
one (and only one) of the following two possibilities occurs:
\begin{itemize}
\item One of the four maps $f_1,f_1^{-1},f_2,f_2^{-1},$ has derivative greater than~$1$
at~$\theta$, while the other three maps have derivative strictly smaller than~$1$ at this point.
\item Two of these maps have derivative equal to~$1$ at~$\theta$, while the other two have
derivative smaller than~$1$ at the same point.
\end{itemize}

From the first remark above and relation~(\ref{deriv}) we deduce that every
point $(u:v)$ which is different from $(0:1)$, $(1:0)$, $(1:1)$, and $(-1:1)$, is
expandable by some element of $G_2$ (and thus of $\PSL_2(\bbZ)$). The latter two
points are expanded by elements in \hspace{0.01cm}
$\PSL_2(\bbZ) \setminus G_2$, \hspace{0.01cm} for instance,
$f=\left[\left(\begin{smallmatrix} 1&-1\\ 0&1\end{smallmatrix}\right)\right]$ and
$g=\left[\left(\begin{smallmatrix} 1&1\\ 0&1\end{smallmatrix}\right)\right]$,
respectively. Since we have already seen that the former points are non-expandable,
this shows that the $\NE$-set for $\PSL_2(\bbZ)$ is reduced to $\{(0:1),(1:0)\}$.

Now notice that the remarks above also show that, among the compositions of length
smaller than or equal to~$n$, the one that expands the most at a generic\footnote{Here,
{\em generic} just means not contained in the orbit by $\PSL_2(\bbZ)$ of $(1:0)$, or
equivalently, the set of $\theta$ for which $\tan(\theta)$ is irrational.} point~$\theta$
can be found by a ``greedy'' algorithm: apply always the generator which expands at the
point obtained after the previous composition.

\begin{lemma}
Given $N \in \mathbb{N}$ and a generic point $\theta$, let $f_{i_1},f_{i_2},\dots,f_{i_n}$
be a finite sequence of elements in $\mcF$ such that $n \leq N$ and such that the 
value of the derivative at the point $\theta$ of the composition 
$f_{i_n} \circ \dots \circ f_{i_2} \circ f_{i_1}$ is maximal among the compositions 
of length smaller than or equal to $N$. Then $n=N$, and the 
composition is obtained by the ``greedy'' algorithm described above.
\end{lemma}

\begin{proof}
We may assume that the composition $f_{i_n} \circ \dots \circ f_{i_2} \circ f_{i_1}$ 
is irreducible, that is, no generator is applied immediately after its inverse. 
Let us denote by $\theta_k$ the image of $\theta$ under the partial composition 
$f_{i_k} \circ \cdots \circ f_{i_2} \circ f_{i_1}$. Notice that if for some $k$ 
the generator which is applied at time $k$ was contracting at the corresponding point
$\theta_{k-1}$ (that is, if $f_{i_k}'(\theta_{k-1})<1$), then the inverse of this generator
would be expanding at the image point $\theta_k$ ({\em i.e.}, $(f_{i_k}^{-1})'(\theta_k) > 1$).
Since for each generic point there is only one generator having derivative greater than one,
and since $f_{i_{k+1}} \neq f_{i_k}^{-1}$, this would imply that $f_{i_{k+1}}'(\theta_k)<1$.
Repeating this argument several times, this would allow us to conclude that, for all
$j\ge k$, one has $f_{i_j}'(\theta_{j-1})<1$. This clearly implies that all the
``tail'' $f_{i_n}\circ\dots f_{i_{k+1}}$ contracts at $\theta_{k}$, and hence
if it is omitted this increases the derivative at the point~$\theta$:
$$(f_{i_n}\circ\dots\circ f_{i_2}\circ f_{i_1})' (\theta)
< (f_{i_{k-1}}\circ\dots\circ f_{i_2}\circ f_{i_1})' (\theta).$$
However, this is in contradiction with our choice 
of the sequence $f_{i_1},f_{i_2},\dots,f_{i_n}$.

Therefore, at each time $k$ the generator which is applied is expanding at the point
$\theta_{k-1}$. In other words, the sequence coincides with the one provided by
the ``greedy'' algorithm.
\end{proof}

\begin{figure}[!h]
    \begin{center}
        \includegraphics{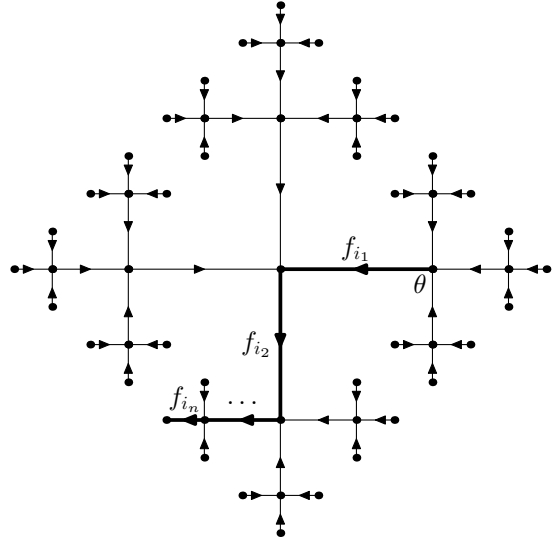}
    \end{center}
    \caption{The Schreier graph of an orbit with the arrows showing the expanding direction}
\end{figure}

The ``greedy'' algorithm reduces the study of the Lyapunov expansion exponent
of~$G_2$ to the study of the Lyapunov exponent of a deterministic dynamics,
namely the one given by applying the map $f_1^{-1}$ on $A_+$, the map $f_1$
on $A_-$, the map $f_2^{-1}$ on $B_+$, and the map $f_2$ on $B_-$. To deal
with this dynamics, let us consider the map $S\!: S^1 \to [0,1]$ obtained
as the ``union'' of the affine charts on $A_{\pm}$, $B_{\pm}$, that is,
$$
S(\theta)=
\begin{cases}
|\tan (\theta)|, & \theta\in A_{-} \cup A_+ = [-\pi/4,\pi/4],\\
|\ctg (\theta)|, & \theta\in B_{-} \cup B_+ = [\pi/4,3\pi/4].
\end{cases}
$$
Since both $S$ and the set $\mcF$ are invariant under conjugacies by the elements 
in the finite group $\mathcal{H} = \{\id, x \mapsto 1/x, x \mapsto -x, x \mapsto -1/x\}$ 
(written in the affine chart $(1:x)$), this dynamics descends to the quotient 
$S^1 / \mathcal{H} = [0,1]$. Actually, a straightforward computation shows 
the following

\begin{proposition}
Given a generic point $\theta \in S^1$, let $f$ be the ``expanding generator''
at this point, that is, the element $f \in \mcF$ such that $f'(\theta)>1$.
Then $S(f(\theta))=\widetilde{\varphi}(S(\theta))$, where
\begin{equation}\label{eq:fractions-def}
\widetilde{\varphi}(x)=\begin{cases}
\frac{1}{\frac{1}{x}-2}, & x\in [0,1/3],\\
\frac{1}{x}-2, & x\in [1/3,1/2],\\
2-\frac{1}{x}, & x\in [1/2,1].
\end{cases}
\end{equation}
\label{prop-cases}
\end{proposition}

In other words, after projecting into the quotient $S^1/ \mathcal{H} = [0,1]$, the 
dynamics of the ``greedy algorithm'' becomes the dynamics of the non-uniformly 
expanding map~$\widetilde{\varphi}$. This map has two neutral fixed 
points (namely $0$ and~$1$), and in analogy to Lemma~\ref{l:varphi} one can
state the following lemma for which we postpone the proof.

\begin{lemma}\label{l:fractions}
For Lebesgue-a.e. point $x\in[0,1]$, the time averages concentrate on
the set $\{0,1\}$. More precisely, for every $\varepsilon>0$ we have
$$
\frac{1}{n} \# \left\{0\le j \le n-1 \mid \widetilde{\varphi}^j(x)\in
U_{\varepsilon}(0) \cup U_{\varepsilon}(1) \right\} \xrightarrow[n\to\infty]{} 1.
$$
\end{lemma}

Since $\widetilde{\varphi}'(0)=\widetilde{\varphi}'(1)=1$, and since 
the function $|\log \widetilde{\varphi}'|$ is bounded on~$[0,1]$ and 
continuous near $0$ and $1$, the lemma above easily implies the following

\begin{corollary}
For Lebesgue-a.e. $x\in [0,1]$, the Lyapunov exponent 
of $\widetilde{\varphi}$ at $x$ is equal to zero.
\end{corollary}

Now to complete the proof of Theorem~\ref{PSL2Z}, notice that the 
derivative of the map $S$ is bounded from above
and away from zero. Thus, by Proposition~\ref{prop-cases}, the 
Lyapunov expansion exponent of $G_2$ (and hence that 
of~$\PSL_2(\bbZ)$) is also equal to zero for Lebesgue-almost 
every point~$\theta$ on the circle.

\vspace{0.15cm}


Now, in the same spirit as that of~Lemma~\ref{l:varphi}, we provide the
\begin{proof}[Proof of Lemma~\ref{l:fractions}]
The first step consists in finding a periodic orbit of period two, say
$\{a,\widetilde{\varphi}(a)\} = \{a,b\}$, such that the interval $J:=[a,b]$ is at the same time 
a fundamental domain for the map in a neighborhood of~$0$ and in a neighborhood 
of~$1$. To do this, we consider the function $\wtp^2$ on the interval 
$(\wtp^{-1}(1/2),1/3)$, where the preimage is taken for the branch $\wtp|_{[0,1/3]}$. 
Then
$$
\wtp^2(\wtp^{-1}(1/2))=\wtp(1/2)=0, \quad \wtp^2(1/3)=\wtp(1)=1,
$$
and since the map $\wtp^2$ in increasing and expanding on
$(\wtp^{-1}(1/2),1/3)$, it has a unique fixed point~$a$ therein.

\begin{figure}[!h]
    \begin{center}
        \includegraphics{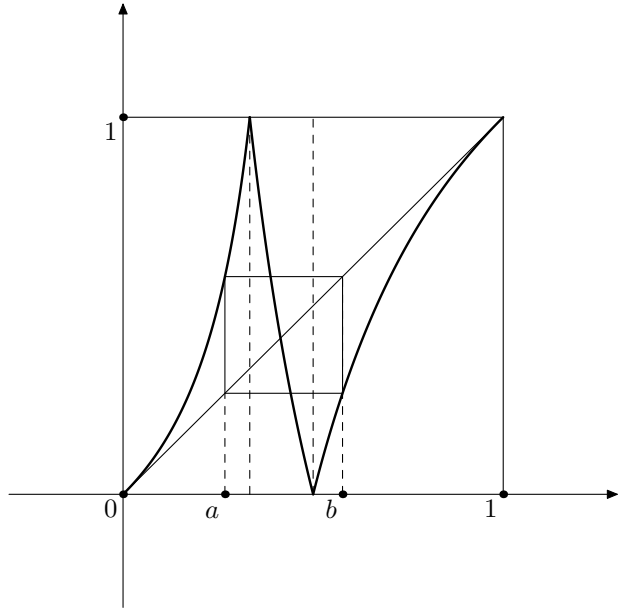}
    \end{center}
    \caption{The order-two periodic point}
\end{figure}

Notice that since $a\in(0,1/3)$ and $b=\wtp(a)\in (1/2,1)$, the
interval $(a,b)$ is simultaneously a fundamental domain for both
$\wtp|_{(0,1/3)}$ and $\wtp|_{(1/2,1)}$ (see Figure~\ref{f:dd2}).

\begin{figure}[!h]
    \begin{center}
        \includegraphics{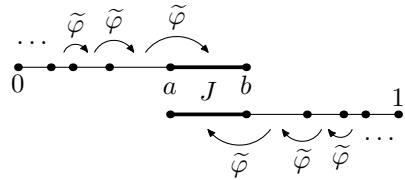}
    \end{center}
    \caption{A simultaneous fundamental domain}\label{f:dd2}
\end{figure}

Now consider the first-return map~$\Phi$ to $J$, as 
well as the return-time function~$\tau$, given by
$$
\Phi(x):=\widetilde{\varphi}^{\tau (x)}(x), \quad
\tau (x):=\min\{n \geq 1 \mid \widetilde{\varphi}^n(x)\in J\}.
$$
The map $\Phi$ can be described in the following way (see Figure~\ref{fig:return}):

\begin{itemize}
\item
The intervals $I_1$ and $I_2$ are mapped onto $[b,1]$, and then they return 
by the topologically repelling map 
$\wtp:[b,1)\to [a,1)$ to the fundamental domain~$J$. They 
are decomposed into infinitely many continuity intervals whose 
images by $\Phi$ coincide with the whole interval $J$.
\item The interval $I_3$ is mapped onto~$J$.
\item
The intervals $I_4$ and $I_5$ are mapped onto $[0,a]$, and then they 
return by the topologically repelling map $\wtp:(0,a]\to (0,b]$ to the fundamental 
domain~$J$. They are decomposed into infinitely many continuity intervals whose 
images by $\Phi$ coincide with the whole interval~$J$.
\end{itemize}

\begin{figure}[!h]
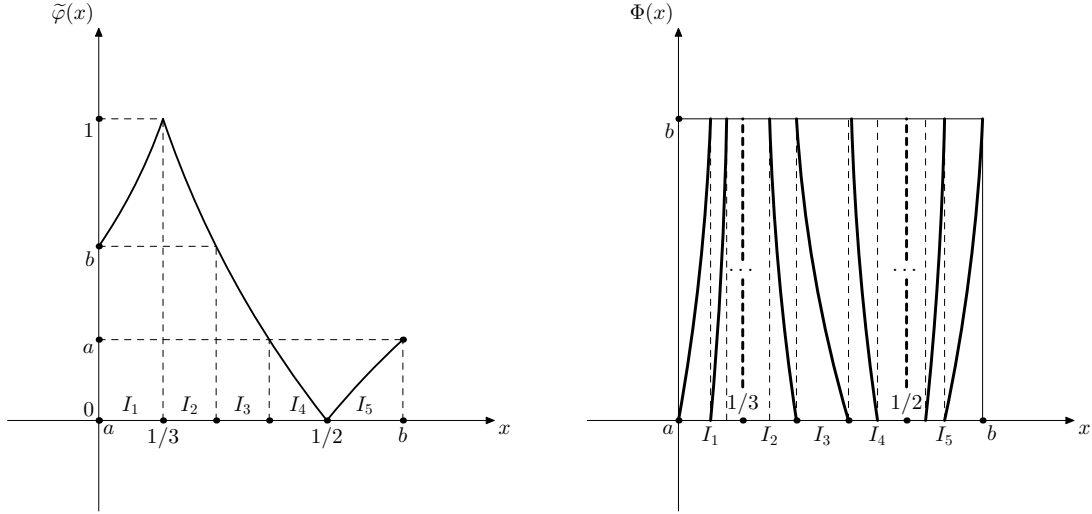

    \begin{center}
        \includegraphics[scale=0.8]{DKN3.3}
        \qquad
        \includegraphics[scale=0.8]{DKN5.2}
    \end{center}
    \caption{The first-return map $\Phi$}\label{fig:return}
\end{figure}

Since $\wtp$ is non-uniformly expanding on the whole interval~$[0,1]$,
and since it is uniformly expanding on~$J$, for all $x \in J$ we have
$$
|\Phi'(x)| \ge |\wtp'(x)| \ge \inf_{y\in J} |\wtp' (y)|>1.
$$
Thus, the map $\Phi$ is uniformly expanding. The following lemma 
provides us a necessary upper bound for the distortion norms 
of the branches of~$\Phi$. We give a more general version 
(which still holds for non-expanding maps) since the underlying 
(simple) idea will be relevant in the next section.

\begin{lemma}\label{l:fixed-exp}
Let $x_0$ be a fixed point of some $C^2$ diffeomorphism $f \!: [x_0,a]\to [x_0,b]$
such that $f(x)>x$ for all $x\in (x_0,a]$. Consider the first-entry map
$F \!: [x_0,a] \to [a,b]$ into the interval $J = [a,b]$, that is
$$
F(x):=f^{k(x)}(x), \quad k(x):= \min\{k \geq 1 \mid f^k (x)\in J \}.
$$
Let $J_k := f^{-k}(J)$ be the (infinitely many) continuity intervals of $F$, and
denote by $f_k$ the restriction of $F$ to $J_k$ (that is, $f_k := f^k|_{J_k}$).
Then the following hold:
\begin{enumerate}
\item There exists a uniform bound for the distortion norms of the maps $f_k$. 
\item Starting from some $k_0$, the maps $f_k$ become uniformly expanding, that is, 
there exists $\lambda>1$ such that, for all $k \ge k_0$ and all $x \in J_k$, one has 
$f_k'(x) \geq \lambda$. Moreover, at the cost of increasing $k_0$, one can take $\lambda=2$.
\end{enumerate}
\end{lemma}

\begin{proof}
First notice that the distortion coefficients of all of the maps~$f_k$
are uniformly bounded. Indeed, since the intervals $J_k$ are disjoint,
\begin{equation}\label{eq:exit-distortion}
\varkappa(f^k;J_k) \le \|\log (f') \|_{C^1} \sum_{i=0}^{k-1} |f^i(J_k)| =
\|\log (f') \|_{C^1} \sum_{i=1}^{k} |J_i| \le {} \| \log (f')\|_{C^1} (b-x_0),
\end{equation}
and since the right hand expression does not depend on~$k$, this
provides the uniform bound for the distortion coefficients. Now
letting $C := \exp \big( \|\log f'\|_{C^1} (b-x_0) \big)$,
the above estimate gives, for all $k \geq 1$ and all $x \in J_k$,
$$
(f^k)'(x)\ge \frac{1}{C} \cdot \frac{|J|}{|J_k|}.
$$
Moreover, since the series $\sum_k |J_k|$ converges, the length $|J_k|$ goes to zero.
In particular, there exists $k_0 \in \mathbb{N}$ such that, for all $k \ge k_0$,
one has \hspace{0.01cm} $|J_k|\le |J| / 2C$. \hspace{0.01cm}
This immediately yields, for all $k \ge k_0$ and all $x \in J_k$,
$$
(f_k)'(x)\ge \frac{|J|}{C|J_k|}\ge 2,
$$
thus proving the second claim of the lemma. To prove the first claim notice that,
according to the control for the distortion coefficients already established,
for each interval $J'\subset J$ and each $i \!\in\! \{1,\ldots,k\}$ we have
$$
|f^{-i}(J')|\le C |J'| \cdot \frac{|J_i|}{|J|}.
$$
Hence,
$$
\varkappa(f_k ; f^{-k}(J')) \le
\|\log f'\|_{C^1} \sum_{i=1}^{k} C |J'| \cdot \frac{|J_i|}{|J|} \le
\frac{C \|\log f'\|_{C^1} (b-x_0)}{|J|}\cdot |J'| = C' |J'|,
$$
which gives $\eta(f_k ; f^{-k} (J')) \leq C'$, thus finishing the proof.
\end{proof}

According to~\cite{mane}, the preceding lemma guarantees the hypotheses which 
ensure the existence of an absolutely continuous ergodic invariant measure 
for~$\Phi$ which is equivalent to the Lebesgue one. Now remark that the return-time 
function $\tau$ is not locally integrable near the points~$1/3$ and~$1/2$ 
(this is due to the fact that these points are mapped by $\widetilde{\varphi}$ into
the parabolic fixed points $0$ and $1$ respectively). Hence, using the very same arguments as those
of the proof of Lemma~\ref{l:varphi}, this allows to finish the proof of Lemma~\ref{l:fractions}.
\end{proof}

We close this section with an explicit construction of conformal measures 
corresponding to Example~\ref{ex:PSLmin}. 

Notice that, due to formula~\eqref{deriv}, if a map $F\in\PSL_2(\bbZ)$ 
sends some point $(m:n)$ into $(a:b)$, where $m,n,a,b$ are integers, and $\gcd(m,n)=\gcd(a,b)=1$, then
$$
F'((m:n))=\left(\frac{\|(m,n)\|}{\|(a,b)\|}\right)^2.
$$
On the other hand, if $\delta>1$, then the sum
$$
S_{\delta}:=\sum_{\gcd(m,n)=1} \frac{1}{\|(m,n)\|^{2\delta}}
$$
is finite. One can then easily see that the measure 
$$
\mu_{\delta}:= \frac{1}{S_{\delta}} \sum_{\gcd(m,n)=1} \frac{\Dirac_{(m:n)}}{\|(m,n)\|^{2\delta}}
$$
is $\delta$-conformal.


\section{Proofs}\label{s:proof}

\subsection{Actions with property~\smin}
\begin{proof}[Proof of Theorem~\ref{description}]
We begin by dealing with the first claim. For this we notice that the set
$\NE (G)$ is closed, since it is a (countable) intersection of closed sets:
$$\NE(G)=\bigcap_{g\in G} \{x\mid g'(x)\le 1\}.$$
Therefore, the finiteness of $\NE(G)$ follows directly from the following 
\begin{lemma}\label{p:isolated}
The set $\NE(G)$ is made up of isolated points.
\end{lemma}
\begin{proof}
For a fixed $y \in \NE(G)$ we will find an interval of the form $(y,y+\delta)$
which does not intersect $\NE(G)$. The reader will notice that a similar argument
provides an interval of the form $(y - \delta',y)$ also disjoint from $\NE(G)$.

By property~\smin, there exist $g_+ \in G$ and $\varepsilon > 0$ such
that $g_+(y)=y$ and such that~$g_+$ has no other fixed point in
$(y,y+\varepsilon)$. Changing $g_+$ by its inverse if necessary, we
may assume~$y$ to be a right topologically repelling fixed point of
$g_+$. Let us consider the point $\bar{y} := y + \varepsilon/2 \in
(y,y+\varepsilon)$, and for each integer $k \geq 0$ let $\bar{y}_k :=
g_+^{-k}(\bar{y})$ and $J_k := (\bar{y}_{k+1},\bar{y}_{k})$. Taking
$a=\bar{y}_1, b=\bar{y}$, and applying Lemma~\ref{l:fixed-exp}, 
we see that for some $k_0 \in \mathbb{N}$ one has $(g_+^k)'(x) \ge 2$ 
for all $k\ge k_0$ and all $x \in J_k$. Hence, $\NE(G) \cap J_k =
\emptyset$, which clearly implies that $\NE(G)\cap
(y,\bar{y}_{k_0})=\emptyset$, thus finishing the proof.
\end{proof}

According to the proof above, for each point $y \in \NE(G)$ 
one can fix an interval $I_y^+ := (y,y+\delta^{+})$ (resp.
$I_{y}^-:=(y-\delta^{-},y)$), a number $k_0^+$ (resp. $k_0^-$), and an
element $g_+$ (resp. $g_-$) in $G$ having~$y$ as a right (resp.
left) topologically repelling fixed point and with no other fixed
point than $y$ in the closure of $I_{y}^+$ (resp. of $I_{y}^-$) and
such that, if for $x \in I_{y}^+$ (resp. for $x \in I_{y}^-$) we
take the smallest integer $n \geq 0$ such that $g_+^n (x) \notin
I_y^+$ (resp. $g_-^n (x) \notin I_y^-$), then $(g_+^{n+k_0^+})'(x)
\geq 2$ (resp. $(g_-^{n+k_0^-})'(x) \geq 2$). We then let
$$
U_{y}:=I_{y}^+ \cup I_{y}^- \cup \{y\}.
$$
By definition (and continuity), for every point $y \notin \NE(G)$ there exist
$g = g_y \in G$ and a neighborhood $V_y$ of $y$ such that $\inf_{V_y} g' > 1$. The
sets $\{U_y \mid y \in \NE(G) \}$ and $\{V_y \mid y \notin \NE(G) \}$ form an open cover of
the circle, from which we can extract a finite sub-cover
$$
\{U_y \mid y \in \NE(G) \} \bigcup \{V_{y_1},\dots, V_{y_k} \}.
$$
Let
$$
\lambda := \min \big\{2, \inf_{V_{y_1}} g_{y_1}', \dots , \inf_{V_{y_k}} g_{y_k}' \big\}.
$$
Since $\lambda$ is the minimum among finitely many numbers greater than~$1$, we have $\lambda>1$.

Now for every $x \in S^1$ either $x \in \NE(G)$ or $x$ lies inside one of the sets $I_y^{\pm}$
or $V_{y_j}$. In the latter case, there exists a map $g\in G$ such that $g'(x) \ge \lambda$, and
we can take the image point $g(x)$ and repeat the procedure. Continuing in this way, we see that
if we do not fall into a point in $\NE(G)$ by some composition, then we can always continue
expanding by a factor at least equal to $\lambda$ by some element in $G$. Therefore, for each
point not belonging to the orbit of $\NE(G)$, the set of derivatives $\{g'(x) \mid g \in G\}$
is unbounded. Since for a point $x$ in the orbit of $\NE(G)$ this set is obviously bounded,
this proves the second claim of Theorem~\ref{description}.

\vspace{0.25cm}

To complete the proof of the theorem, the only conclusion which is left corresponds
to that of the ergodicity of the action. Thus, let $A \subset S^1$ be an invariant
measurable set of positive Lebesgue measure, and let~$a$ be a density point in~$A$
not belonging to the orbit of~$\NE(G)$ (notice that, since this orbit is countable,
such a point $a$ exists). Then the expansion procedure works by applying the
``exit-maps'' $g_{\pm}^{n+k_0^{\pm}}$ 
to points in $I_y^+ \cup I_y^-$ and the map $g_{y}$ to points
in $V_{y}$. Now what we need to do is to control the distortion of these compositions in
a small neighborhood of $a$ until its image reaches a ``macroscopic'' length. Although
this can be done in terms of distortion coefficients, we prefer working directly with
the derivatives of the maps which are involved, since this approach provides another 
way to deal with the ergodicity conjecture and allows to state later an interesting problem,
namely Question~\ref{q:C-expansion}.

To simplify, in what follows to our prescribed system of generators we add 
the elements of the form $g_+$ and $g_{-}$, as well as their inverses. The 
main issue below consists in controlling the sum of the derivatives along 
a sequence of compositions by the derivative of the whole composition.

\begin{lemma}\label{l:one-step}
There exists a constant $C_1>0$ such that, for every $x\notin \NE (G)$, one can find
a composition $f_n\circ\dots \circ f_1$ of elements $f_1,\dots,f_n$ in $\mcF$ such
that $(f_n\circ\dots\circ f_1)'(x)\ge \lambda$ and
\begin{equation}\label{eq:sum-ld}
\frac{\sum_{j=1}^n (f_j\circ\dots\circ f_1)'(x)}{(f_n\circ\dots\circ
f_1)'(x)} \le C_1.
\end{equation}
\end{lemma}

\begin{proof}
A compactness type argument reduces the general case to the study of points 
in (arbitrarily small) neighborhoods of $\NE$. We can work in a right neighborhood 
(the case of a left neighborhood is similar) of some point $y\in \NE$. We will use 
the notation of the proof of Lemma~\ref{p:isolated}: take the map $g_+\in G$
that has $y$ as a right-repelling fixed point, a point $\bar{y}$ within the 
right basin of repulsion of~$y$, and denote 
$$
\bar{y}_k:=g_+^{-k}(y), \quad J_0:=[\bar{y}_1,\bar{y}), \quad J_k:=g_+^{-k}(J_0). 
$$
We know that for some $k_0^+$ one has $(g_+^n)'(x)\ge \lambda$ for all $n\ge  k_0^+$ 
and all $x\in J_n$.
So, for every $x\in I_y^+:=(y,\bar{y}_{k_0^+})$ we take $n \in \mathbb{N}$ such that 
$x \in J_n$, and we put $f_1 = \ldots = f_n := g_+$. It suffices now to estimate the 
quotient~\eqref{eq:sum-ld}. To do this, we notice that Proposition~\ref{p:sum} easily
implies that
$$
\sum_{j = 1}^n (g_+^j)'(x) \leq \exp(C_{\mathcal{F}}) \cdot
\frac{\sum_{j=1}^{n} |g_+^j (J_n)|}{|J_n|} \leq
\frac{\exp(C_{\mcF})}{|J_n|}$$ and
$$
(g_+^n)'(x) \geq \exp(-C_{\mathcal{F}}) \cdot \frac{|g^n
(J_n)|}{|J_n|} = \exp(-C_{\mathcal{F}}) \cdot \frac{|J_0|}{|J_n|}.
$$
Therefore,
$$
\frac{\sum_{j=1}^n (f_j\circ\dots\circ f_1)'(x)}{(f_n\circ\dots\circ f_1)'(x)}
= \frac{\sum_{j = 1}^n (g_+^j)'(x)}{(g_+^n)'(x)}
\leq \frac{\exp(2 C_{\mathcal{F}})}{|J_0|}.
$$
\end{proof}

The previous lemma provides us a natural ``expansion procedure'' which yields the following
\begin{lemma}\label{l:derivatives-quotient}
There exists a constant $C_2$ such that, for every point $x$ which does not belong to the
orbit of $\NE(G)$ and every $M > 1$, one can find $f_1,\ldots,f_n$ in $\mcF$ such that the
composition $f_n \circ \dots \circ f_1$ has derivative greater than or equal to $M$ at~$x$ and
\begin{equation}\label{eq:sums-ld-2}
\frac{\sum_{j=1}^n (f_j\circ\dots\circ f_1)'(x)}{(f_n\circ\dots\circ f_1)'(x)}\le C_2.
\end{equation}
\end{lemma}

\noindent{\em Proof.} Starting with $x_0 = x$ we let
$$
x_k := f_{k,n_k}\circ\dots\circ f_{k,1} (x_{k-1}),
$$
where the elements $f_{k,j} \in \mcF$ (chosen using Lemma~\ref{l:one-step}) satisfy
$$
\frac{\sum_{j=1}^{n_k} (f_{k,j} \circ\dots\circ f_{k,1})'(x_{k-1})}{(f_{k,n_k}\circ\dots\circ f_{k,1})'(x_{k-1})}
\le C_1, \qquad (f_{k,n_k} \circ\dots\circ f_{k,1})'(x_{k-1}) \ge \lambda.
$$
If we perform this procedure $K \geq \log(M) / \log(\lambda)$
times, then for the compositions
$$
F_k = (f_{k,n_k} \circ\dots\circ f_{k,1}) \circ \dots\circ (f_{1,n_1} \circ\dots\circ f_{1,1})
$$
we obtain
$$
F_K'(x) = \prod_{k=1}^K (f_{k,n_k}\circ\dots\circ f_{k,1})'(x_{k-1}) \ge \lambda^K \geq M.
$$

For estimating the quotient in the left hand side expression of~\eqref{eq:sums-ld-2}, we 
will write it differently. Namely, letting $y := f_n\circ\dots\circ f_1 (x)$, one
can easily check that
\begin{equation}\label{eq:sum-inverses}
\frac{\sum_{j=1}^n (f_j\circ\dots\circ f_1)'(x)}{(f_n\circ\dots\circ
f_1)'(x)} = \sum_{j=1}^n (f_{j+1}^{-1}\circ \dots f_n^{-1})'(y).
\end{equation}
In other words, providing a control for the quotient in~\eqref{eq:sums-ld-2} 
corresponding to an expansion at $x$ is equivalent to providing a control 
for the sum of the derivatives for the contraction at $y$. 

Now, to simplify the notation, we will denote by $\widetilde{F}_k$ the 
composition obtained at each step of the expansion procedure, that is, 
$$
\widetilde{F}_k := f_{k,n_k} \circ\dots\circ f_{k,1}.
$$
If we denote $y := F_K (x)$, then using~\eqref{eq:sum-inverses} we see 
that the left hand side expression in~\eqref{eq:sums-ld-2} is equal to
\begin{multline*}
\sum_{k=1}^K \sum_{j=1}^{n_k} \left((f_{k,j+1}^{-1} \circ \dots
\circ f_{k,n_k}^{-1})\circ (\widetilde{F}_{k+1}^{-1}\circ\dots\circ
\widetilde{F}_K^{-1})\right)' (y) =
\\
= \sum_{k=1}^K (\widetilde{F}_{k+1}^{-1}\circ\dots\circ
\widetilde{F}_K^{-1})' (y) \cdot \sum_{j=1}^{n_k} (f_{k,j+1}^{-1}
\circ \dots \circ f_{k,n_k}^{-1})'(x_{k}) \le
\\
\leq \sum_{k=1}^{K} \frac{1}{\lambda^{K-k}} \cdot C_1 \le
\frac{C_1}{1-\lambda^{-1}}. \qquad \square
\end{multline*}


Finally, the bound obtained in the previous lemma provides the desired control
of distortion for the compositions on a neighborhood which is expanded up to
a macroscopic length. More precisely, the following holds.


\begin{proposition}\label{p:final}
There exists $\varepsilon > 0$ such that, for every point $x$ not
belonging to the orbit of $\NE (G)$, there exists a sequence $V_k$ of
neighbourhoods of $x$ converging to $x$ such that to each $k \in
\mathbb{N}$ one can associate a sequence of elements
$f_1,\ldots,f_{n_k}$ in $\mcF$ satisfying $|f_{n_k} \circ \dots \circ f_1
(V_k)| = \varepsilon$ and $\varkappa (f_{n_k} \circ \dots \circ f_1 ;
V_k) \leq \log (2)$.
\end{proposition}

\begin{proof}
We will check the conclusion of the lemma for \hspace{0.01cm}
$\varepsilon = \log(2) / (2 C_{\mcF} C_2)$.  Indeed,
fix $M > 1$ and consider the sequence of compositions $f_n \circ
\dots \circ f_1$ associated to $x$ and $M$ provided by the previous
lemma. Denoting $\bar{F}_n := f_n \circ \dots \circ f_1$ and
$y:=\bar{F}_n(x)$, for the neighborhood $V := \bar{F}_n^{-1}
(U_{\varepsilon/2} (y))$ of $x$ we have 
$$
\varkappa \big( \bar{F}_n ; \bar{F}_n^{-1} (U_{\varepsilon/2} (y))
\big) = \varkappa \big( \bar{F}^{-1}_n ; U_{\varepsilon/2} (y) \big)
= \varkappa \big( f_1^{-1} \circ \dots \circ f_n^{-1},
U_{\varepsilon/2} (y) \big).
$$
By Proposition~\ref{bound}, the distortion coefficient of the
composition $f_1^{-1} \circ \dots \circ f_n^{-1}$ is bounded from
above by $\,\log (2)\,$ in a neighborhood of~$y$ of radius
$$
r := \frac{\log(2)}{4 C_{\mcF} S},
$$
where
\begin{equation}\label{eq:F-inverses}
S := \sum_{j=1}^{n} (f_{j+1}^{-1} \circ \dots \circ f_n^{-1})'(\bar{F}_n (x)).
\end{equation}
Now according to~\eqref{eq:sum-inverses}, the conclusion~\eqref{eq:sums-ld-2} allows to 
estimate the sum~\eqref{eq:F-inverses}:
$$
S = \frac{\sum_{j=1}^{n} (f_j\circ\dots\circ f_1)'(x)}{(f_n \circ\dots\circ
f_1)'(x)}  \le C_2.
$$
Therefore, the chosen $\varepsilon$ is less than or equal to~$2r$, which implies 
the desired estimate for the distortion. Finally, notice that
$$
|V|=\big| \bar{F}_n^{-1} (U_{\varepsilon/2} (y)) \big|
\leq \frac{|U_{\varepsilon/2} (y)| \hspace{0.01cm} \exp \big(
\varkappa \big( \bar{F}_n^{-1}; U_{\varepsilon/2} (y) \big) \big)}{
(f_n \circ \dots \circ f_1)'(x)} \leq \frac{2 \varepsilon}{M},
$$ 
and the last expression tends to zero as $M$ goes to infinity. Thus, for the 
family of neighborhoods $V=V_k$ obtained by the procedure described 
above for $M=k$ going to infinity, we indeed see that $V_k$ indeed collapse to~$x$, 
and this concludes the proof of the proposition.
\end{proof}

The latter proposition provides the desired bound for the distortion of the expansion on small 
neighborhoods of the density point~$a \in A$. By the arguments already mentioned 
in Section~\ref{ss:minimality}, this implies the ergodicity of the action, thus finishing the proof of Theorem~\ref{description}.
\end{proof}

\begin{proof}[Proof of Corollary~\ref{cor:independence}] Let $G$ be a finitely generated group of $C^2$ circle diffeomorphisms for
which property \smin\ holds with respect to a prescribed Riemannian metric. Given a new Riemannian
metric on $S^1$, let us denote by $c \!: S^1 \to \bbR$ the function which corresponds to
the quotient of both metrics. If we denote by $g'$ (resp. $g^{\bullet}$) the derivative of
$g \in G$ with respect to the original (resp. the new) metric, then one has
\hspace{0.01cm} $g^{\bullet} (x) = \frac{c(g(x))}{c(x)} g'(x)$. \hspace{0.01cm}
By Theorem~\ref{description}, a point $x \in S^1$ does not belong to the orbit of any non-expandable
point with respect to the initial metric if and only if the set $\{g'(x) \mid g \in G\}$
is unbounded. Now since the value of $c$ is bounded from 
above and away from zero, this happens if and only if the set
$\{g^{\bullet}(x) \mid g \in G\}$ is also unbounded. Therefore, every point $x$ which
is non-expandable for the new metric is in the orbit of some point $x_0$ which is 
non-expandable for the original one. By property \smin, there exists $g_-$ and $g_+$ in $\Gamma$
having $x_0$ as an isolated fixed point by the left and by the right, respectively.
Choosing $h \in G$ such that $x = h(x_0)$, this implies that $x$ is a fixed point
which is isolated by the left (resp. by the right) for $h g_- h^{-1}$ (resp.
$h g_+ h^{-1}$), and this shows that property \smin\ holds with respect to the new metric.
\end{proof}


We would like to close this section with a few comments on the idea of the 
proof of Theorem~\ref{description} above. 
For this, let us first introduce some terminology.

\begin{definition}
For $C > 0$ a point $x\in S^1$ is said to be $C$-\emph{distortion-expandable} for the
action of a finitely generated group~$G$ of $C^2$ circle diffeomorphisms if for each $M>1$
one can find $f_1,\ldots,f_n$ in $\mcF$ such that $(f_n \circ\dots\circ f_1)'(x) \geq M$
and
$$
\frac{\sum_{j=1}^{n} (f_j\circ\dots\circ f_1)'(x)}{(f_n\circ\dots\circ f_1)'(x)} \leq C.
$$
\end{definition}

The arguments of the proof of Proposition \ref{p:final} prove more generally that
if for some $C > 0$ the set of $C$-distortion-expandable points has positive Lebesgue
measure and the action is minimal, then the Lebesgue measure of this set equals $1$, and
the action is ergodic. Therefore, a positive answer for the following question would also
provide a positive answer for the ergodicity conjecture.

\begin{question}\label{q:C-expansion}
Is it true that, for every finitely generated non Abelian group of $C^2$ circle
diffeomorphisms whose action is minimal, there exists a constant $C>0$ such that
Lebesgue-a.e. point is $C$-distortion-expandable~?
\end{question}

\subsection{Conformal measures}

\begin{proof}[Proof of Theorem~\ref{ConformalStar}]
We will use the following fact from basic Measure Theory:
\begin{proposition}\label{p:measure}
For any two non-atomic measures $\mu_1$ and $\mu_2$ on the circle, 
the limit (in $[0,\infty]$)
\begin{equation}\label{eq:density}
\rho(x)=\rho_{\mu_1,\mu_2}(x)=
\lim_{\varepsilon\to 0} \frac{\mu_1(U_{\varepsilon}(x))}{\mu_2(U_{\varepsilon}(x))}
\end{equation}
exists for $(\mu_1+\mu_2)$-almost every~$x$. This limit is nonzero for $\mu_1$-almost 
every~$x$ and finite for $\mu_2$-almost every~$x$. The set 
$A_0:=\{x\mid \rho(x)=0\}$ (resp., $A_{\infty}:=\{x\mid \rho(x)=\infty\}$) correspond to 
the singular part of $\mu_2$ with respect to $\mu_1$ (resp. of $\mu_1$ with resp. to $\mu_2$). The restrictions of the measures $\mu_1$ and $\mu_2$ to the set $B:=\{0<\rho(x)<\infty\}$ 
are equivalent, and the density of $\mu_1$ w.r.t. $\mu_2$ on $B$ equals~$\rho$.
\end{proposition}

Let us first consider the case of a minimal case (as we will see, the very 
same arguments can be used in the case of an exceptional minimal set).
Let $\mu$ be a conformal measure with some exponent~$\delta$ for an 
action satisfying the property~\smin, and assume that $\mu$ does not 
charge the orbit of $\NE(G)$.
Notice that an atom of~$\mu$ can be placed only at a point~$x$ with 
a bounded set of derivatives~$\{g'(x)\mid g\in G\}$. But we know from 
the second conclusion of Theorem~\ref{description} that such a point 
must belong to the orbit of a non-expandable one. So, as we assumed 
that the orbit of $\NE(G)$ is not charged, the measure~$\mu$ is non-atomic.

Now, let us take any point $x\notin G(\NE)$ and let us analyze the behaviour 
of the limit~(\ref{eq:density}) with the help of Proposition~\ref{p:final}. This proposition 
provides us a sequence of neighborhood $U_k$ of the point $x$, as well as 
expanding compositions $F_k:=f_{n_k}\circ\dots\circ f_1$, one has $\varkappa (F_k;
U_k) \leq \log (2)$. Hence, for every $y\in U_k$
\begin{equation}\label{eq:derivative-estimate}
\frac{|F_k(U_k)|}{2|U_k|}\le F_k'(y)\le \frac{2|F_k(U_k)|}{|U_k|}.
\end{equation}
From the definition of $\delta$-conformality it follows that
$$
\frac{1}{2^{\delta}} \cdot \left(\frac{|F_k(U_k)|}{|U_k|}\right)^{\delta} \mu(U_k) \le \mu(F_k(U_k)) \le 2^{\delta} \cdot \left(\frac{|F_k(U_k)|}{|U_k|}\right)^{\delta} \mu(U_k),
$$
and hence,
$$
\frac{1}{2^{\delta}} \cdot \left(\frac{|U_k|}{|F_k(U_k)|}\right)^{\delta} \mu(F_k(U_k)) \le \mu(U_k) \le 2^{\delta} \cdot \left(\frac{|U_k|}{|F_k(U_k)|}\right)^{\delta} \mu(F_k(U_k)).
$$
Since the length of the image $F_k(U_k)$ equals~$\varepsilon$ for every~$k$, 
the measures $\mu(F_k(U_k))$ are bounded from below independently 
on~$k$. Thus, we have 
\begin{equation}\label{eq:conformal-estimate}
c |U_k|^{\delta} \le \mu(U_k) \le C |U_k|^{\delta}
\end{equation}
for some constants $C>c>0$.

Now, let us consider the three possible cases for the conformal exponent: $\delta<1$, 
$\delta=1$, and $\delta>1$. 
In the first case, 
we have
$$
\lim_{k\to\infty} \frac{\mu(U_k)}{|U_k|} \ge \lim_{k\to\infty} \frac{c|U_k|^{\delta}}{|U_k|} =\infty.
$$
So, for a subsequence of neighborhoods surrounding~$x$ (recall that the neighborhoods~$U_k$ provided by Proposition~\ref{p:final} are not of arbitrary size, though they collapse to~$x$), the 
``density'' limit~\eqref{eq:density} is infinite. 
But due to Proposition~\ref{p:measure}, this limit should be finite for Lebesgue-almost every point~$x$. This contradiction shows that this case is impossible.

On the other hand, if $\delta>1$, 
$$
\lim_{k\to\infty} \frac{\mu(U_k)}{|U_k|} \le \lim_{k\to\infty} \frac{C |U_k|^{\delta}}{|U_k|} =0.
$$
However, according to Proposition~\ref{p:measure}, the limit should be positive 
for $\mu$-almost every~$x$. Since by assumption the measure $\mu$ does not charge 
the set~$G(\NE)$, this gives a contradiction which makes this case impossible. 

The only case which is left is $\delta=1$. But in this case the estimate~\eqref{eq:conformal-estimate} implies that the measure $\mu$ is absolutely continuous with respect to the Lebesgue one, 
and its density (due to the fact that $\mu$ is 1-conformal) is an invariant function. Since the 
Lebesgue measure is ergodic, this density is constant, and hence the measure~$\mu$ is 
proportional to the Lebesgue one, and actually equal to it due to the normalization. This 
concludes the proof in the minimal case.

Assume now that the group $G$ acts with an exceptional minimal set~$\Lambda$ 
and satisfies property~\Ls. Once again, we see that if a conformal measure does not 
charge~$G(\NE)$, then it must be non-atomic. We still have the same 
estimates~\eqref{eq:derivative-estimate} 
and~\eqref{eq:conformal-estimate} on the 
derivative and on the quotient of measures, though the neighborhoods are now considered 
only for points in~$\Lambda$. The argument excluding the exponent $\delta>1$ 
still works: the density limit $\rho_{\mu,\Leb}$ of the measure $\mu$ w.r.t. 
the Lebesgue one 
cannot be zero $\mu$-almost everywhere. 

Similar arguments to the above ones exclude the case~$\delta=1$: indeed, if $\delta$ was equal to~$1$, this would imply that the measure~$\mu$ is absolutely continuous with respect to the Lebesgue one, which is impossible, since the set~$\Lambda$ has zero Lebesgue measure.

Finally, the case $\delta<1$ becomes possible: the density limit of~$\mu$ with respect to the Lebesgue measure will be infinite only at the points of~$\Lambda$. However, there can be only one conformal exponent~$\delta$ and only one conformal measure~$\mu$ corresponding to this exponent. Indeed, let $\delta_1\ge \delta_2$ be two conformal exponents corresponding to conformal measures~$\mu_1$ and $\mu_2$, respectively. Then, by re-applying the same arguments of control of distortion as above, and noticing that the measures $\mu_1,\mu_2$ of an interval $F_k(U_k)$ of length~$\varepsilon$ are bounded from below, we deduce from~\eqref{eq:derivative-estimate} that
\begin{equation}\label{eq:deltas}
c |U_k|^{\delta_1-\delta_2} \le \frac{\mu_1(U_k)}{\mu_2(U_k)} \le C |U_k|^{\delta_1-\delta_2}
\end{equation}
for some constants $C>c>0$.

If $\delta_1>\delta_2$, the inequalities~\eqref{eq:deltas} imply that the 
density limit~\eqref{eq:density} for the measures $\mu_1,\mu_2$ is zero 
on a subsequence for every point $x\in \Lambda\setminus G(\NE)$. However, 
this is impossible, since this density limit should be positive for $\mu_1$-almost 
every point of $\supp(\mu_1)=\Lambda$. Finally, if $\delta_1=\delta_2$, these 
conformal measures are equivalent, 
and the density $\frac{d\mu_1}{d\mu_2}$ is an invariant function. So, once 
we prove that the measure $\mu_1$ is ergodic, this will imply that~$\mu_1=\mu_2$.

The ergodicity of the measure $\mu_1$ can be deduced in the same way 
as in Theorem~\ref{description} was deduced the 
ergodicity of the Lebesgue measure for minimal case. 
Namely, if~$A\subset\Lambda$ is a measurable invariant set, 
then $\mu_1$-almost every point in~$A$ is its $\mu_1$-density point. 
By expanding arbitrarily small neighborhoods of such 
a point, using the fact that (due to the minimality of the action on~$\Lambda$) 
one has $\supp(\mu_1)=\Lambda$, and choosing a subsequence among the 
expanded intervals, at the limit we obtain an interval on 
which the points of~$A$ form a subset of full~$\mu_1$-measure. Due to the 
minimality, this implies that~$A$ has full $\mu_1$-measure. This concludes 
the proof of the ergodicity, and thus that of Theorem~\ref{ConformalStar}.
\end{proof}

\subsection{Random dynamics}

\begin{proof}[Proof of Theorem~\ref{thm:positive}]
Let $m$ be a measure on~$G$ having finite first word-moment and
such that there is no measure on the circle which is simultaneously invariant 
by all the maps in~$\supp(m)$. By P.~Baxendale's theorem (see Section~\ref{ss:Random}), 
there exists an ergodic stationary measure~$\nu$ such that the 
corresponding random Lyapunov exponent is strictly negative. 
We will prove that for $\nu$-almost every point~$x$ the Lyapunov 
expansion exponent at~$x$ is positive. More precisely, we will prove 
that
\begin{equation}\label{eq:L-estimate}
\lambda_{exp}(\nu;G;\mathcal{F}) \ge \frac{|\lambda_{RD}(\nu;m)|}{v_{\mcF}(m)},
\end{equation}
where $\lambda_{exp}(\nu;G;\mathcal{F})$ stands for the Lyapunov expansion exponent at $\nu$-almost every point (due to the ergodicity of the measure~$\nu$, this exponent is constant $\nu$-almost everywhere), and $v_{\mcF}(m)$ denotes the \emph{rate of escape} for the convolutions of $m$:
$$
v_{\mcF}(m)=\lim_{n\to\infty} \frac{1}{n} \int_{G^n} \|g_1\circ \dots \circ g_n \|_{\mcF} \, dm(g_1)\dots dm(g_n).
$$
Notice that a direct consequence of~\eqref{eq:L-estimate} is that 
$$
\lambda_{exp}(\nu;G;\mathcal{F}) \ge \frac{|\lambda_{RD}(\nu;m)|}{\int_G \|g\|_{\mcF} \, dm(g)}.
$$

To prove the estimate~\eqref{eq:L-estimate}, fix $\varepsilon>0$, and 
consider the skew-product map
$$
F:S^1\times G^{\bbN}\to S^1\times G^{\bbN}, \quad F(x,(g_i))=(g_1(x),(g_{i+1})).
$$
Since $\nu$ is an ergodic stationary measure, 
the Random Ergodic Theorem (see, \emph{e.g.},~\cite{Furman}) asserts that 
the $F$-invariant measure $\widetilde{\nu}=\nu\times m^{\bbN}$ is ergodic. 

For each $n\in\bbN$, consider the sets 
$$
A_n:=\{(x,(g_i)) \mid \log (g_n\circ \dots \circ g_1)'(x) < -n(|\lambda_{RD}(m;\nu)|-\varepsilon) \},
$$
and
$$
B_n:=\{(x,(g_i)) \mid \|g_n\circ\dots\circ g_1\|_{\mcF} < n(v_{\mcF}(m)+\varepsilon) \}.
$$
Notice that the measures of both sets $A_n$ and $B_n$ tend to~$1$ as $n$ tends to infinity 
(this follows immediately from the definitions of the random Lyapunov exponent and 
of the rate of escape). Clearly, the same holds for the measures of the 
sets $F^n(A_n\cap B_n)$ (as $F$ preserves the measure~$\widetilde{\nu}$), as
well as for the $\nu$-measures of the projections of these sets on the circle. 
But a point~$y$ belongs to the projection
$C_n:=\pi_{S^1}(F^n(A_n\cap B_n))$ if and only if there exist~$x,g_1,\dots,g_n$ such that 
$$
y=(g_n\circ\dots\circ g_1)(x), \quad \log (g_n\circ \dots \circ g_1)'(x) < -n(|\lambda_{RD}(m;\nu)|-\varepsilon),
$$
$$
\text{and} \quad \|g_n\circ\dots\circ g_1\|_{\mcF} < n(v_{\mcF}(m)+\varepsilon). 
$$
This implies that for the composition $g_1^{-1}\circ \dots\circ g_n^{-1}$ one has 
$$
\frac{\log (g_1^{-1}\circ \dots\circ g_n^{-1})'(y)}{\|g_1^{-1}\circ \dots\circ g_n^{-1}\|_{\mcF}} \ge \frac{|\lambda_{RD}(m;\nu)|-\varepsilon}{v_{\mcF}(m)+\varepsilon}.
$$
As $\nu(C_n)\to 1$, the set of points $y$ belonging to an infinite number of sets $C_n$ is of full $\nu$-measure. Hence, 
$$
\lambda_{\exp}(\nu;G;\mcF)\ge \frac{|\lambda_{RD}(m;\nu)|-\varepsilon}{v_{\mcF}(m)+\varepsilon},
$$
and since $\varepsilon>0$ was arbitrary,
$$
\lambda_{\exp}(\nu;G;\mcF)\ge \frac{|\lambda_{RD}(m;\nu)|}{v_{\mcF}(m)},
$$
which concludes the proof of the theorem.
\end{proof}
As we already noticed in Remark~\ref{r:limit}, assuming more restrictive assumptions 
on the moments, one can prove that the ``exponentially expanding'' compositions can 
be chosen of any length. To do this, due to Borel-Cantelli Lemma, it suffices to check 
that the series $\sum_n (1-\nu(C_n))$ converges. And indeed, by establishing some the 
control for the ``large deviations'', one can show (under certain additional assumptions) 
that the measures of the sets $A_n$ and $B_n$ tend to~$1$ 
exponentially, which immediately implies this convergence. 

%
%

\section{Acknowledgments}

The authors would like to express their gratitude to \'Etienne Ghys and Yulij Ilyashenko for
many fruitful discussions and their interest in the problem. The second author would also like
to thank Frank Loray, Dominique Cerveau, Nicolas Monod, Pierre de la Harpe, and 
Vadim Kaimanovich, for valuable remarks and suggestions.

All the authors would like to thank the Universities of Geneva and Santiago de Chile for their hospitality,
and to acknowledge the partial support of the Swiss National Science Foundation as well as the
PBCT-Conicyt for the support via the Research Network on Low Dimensional Dynamics. The second
author was also partially supported by the Russian Foundation for Basic Research grants
7-01-00017-a and CNRS-L\_a 05-01-02801.


\begin{small}

\vspace{0.1cm}

Bertrand Deroin

Universit\'e Paris-Sud, Lab. de Math\'ematiques, B\^at 425

91405 Orsay Cedex, France

Bertrand.Deroin@math.u-psud.fr

\vspace{0.3cm}

Victor Kleptsyn

Institut de Recherches Math\'ematiques de Rennes (UMR 6625 CNRS)

Campus Beaulieu, 35042 Rennes, France

Victor.Kleptsyn@univ-rennes1.fr

\vspace{0.3cm}

Andr\'es Navas

Universidad de Santiago de Chile

Alameda 3363, Santiago, Chile

anavas@usach.cl

\end{small}

\end{document}